\documentclass[11pt,a4paper,reqno]{amsart}
\usepackage{epsfig}
\usepackage{color}
\usepackage{amssymb,amsmath,amsthm,amstext,amsfonts}
\usepackage[normalem]{ulem}
\usepackage{amsmath,amscd,amstext,amsthm,amsfonts,latexsym}
\usepackage[colorlinks=true,linkcolor=blue, urlcolor=red, citecolor=blue, backref]{hyperref}	

\usepackage{amsmath,amssymb,amsthm}
\usepackage{color}
\usepackage[active]{srcltx}
\usepackage{a4wide}
\usepackage{amssymb, amsmath}
\usepackage{graphicx}
\usepackage{float}
\usepackage[active]{srcltx}

\def \cP {\mathcal{P}}
\def \N {\mathbb{N}}
\def \vep {\varepsilon}

\theoremstyle{plain}

\newtheorem{theorem}{Theorem}[section]
\newtheorem{lemma}[theorem]{Lemma}
\newtheorem{proposition}[theorem]{Proposition}
\newtheorem{corollary}[theorem]{Corollary}

\newtheorem{maintheorem}{Theorem}

\newtheorem{example}[theorem]{Example}

\newtheorem{definition}[theorem]{Definition}

\numberwithin{equation}{section}

\newcommand{\intav}[1]{\mathchoice {\mathop{\vrule width 6pt height 3 pt depth  -2.5pt
\kern -8pt \intop}\nolimits_{\kern -6pt#1}} {\mathop{\vrule width
5pt height 3  pt depth -2.6pt \kern -6pt \intop}\nolimits_{#1}}
{\mathop{\vrule width 5pt height 3 pt depth -2.6pt \kern -6pt
\intop}\nolimits_{#1}} {\mathop{\vrule width 5pt height 3 pt depth
-2.6pt \kern -6pt \intop}\nolimits_{#1}}}

\title[Dynamical metric order]{Dynamical metric order}

\author[M. Carvalho]{Maria Carvalho}
\address{CMUP \& Departamento de Matemática, Faculdade de Ci\^encias da Universidade do Porto, Rua do Campo Alegre s/n, 4169-007 Porto, Portugal.}
\email{mpcarval@fc.up.pt; Orcid ID: http://orcid.org/0000-0001-6929-6442}

\author[F. B. Rodrigues]{Fagner B. Rodrigues}
\address{Departmento de Matemática, Universidade Federal do Rio Grande do Sul, Brazil}
\email{fagnerbernardini@gmail.com; Orcid ID:  https://orcid.org/0000-0002-7596-8214}

\begin{document}

\date{\today}

\begin{abstract}
We introduce a counterpart of the metric mean dimension with the role of the box-counting dimension being played by the metric order. This new concept is devised for continuous maps acting on compact metric spaces with infinite box-counting dimension but finite metric order. It provides further information about full shifts with infinite metric mean dimension; and yields a sharper estimate of complexity for the maps induced on the space of probability measures by continuous transformations on compact metric spaces, whose upper metric mean dimension is known to admit only two values (zero or infinity). When applied to homeomorphisms, it satisfies a variational principle where the maximization is taken over the space of invariant probability measures and whose equilibrium states always exist.
\end{abstract}

\maketitle



\section{Introduction}

Let $(X,d)$ be a compact metric space and $f \colon X \to X$ be a continuous map. Denote by $\mathcal{P}(X)$ the space of Borel probability measures on $X$ endowed with the weak$^*$-topology, by $\mathcal{P}_f(X)$ the subspace of $f$-invariant measures and by $\mathcal{E}_f(X)$ its subset of ergodic elements. The space $\mathcal{P}(X)$ is nonempty, and it is also compact when endowed with the weak$^*$-topology. There are metrics on $\mathcal{P}(X)$ that induce this topology, the classic ones being the \emph{Wasserstein distances} $W_p$, for $1 \leqslant p < +\infty$, and the \emph{L\'evy-Prokhorov distance} $LP$ (see Subsection~\ref{sse:metric} for the precise definitions). The map $f$ induces in a natural way a continuous transformation $f_\ast \colon \mathcal{P}(X) \to \mathcal{P}(X)$ defined by $f_\ast(\mu) = \mu \circ f^{-1}$, where for every Borel set $B$ one has $\big(\mu \circ f^{-1}\big) (B) = \mu (f^{-1}(B))$.
\smallskip

Dimension theory studies the geometrical complexity of spaces by measuring sets using, for instance, open covers with varying diameters and estimating the exponential growth of their covering numbers with the approximating scale. Although the list of options is large, we consider the (upper) \emph{box-counting dimension}, defined by
$$\mathrm{\overline{\dim}_B}(X,d) \, = \, \limsup_{\vep \, \to \, 0^+}\, \frac{\log N(X,d,\vep)}{|\log \vep|}$$
where, given $\vep > 0$, $N(X,d,\vep)$ stands for the minimum number of open balls with radius $\vep$, in the metric $d$, that cover $(X,d)$. When the box-counting dimension is infinite, one may need to employ a refined concept to distinguish spaces; we resort to the (upper) \emph{metric order}, defined by
$$\overline{\mathrm{mo}}\,(X,d) \, = \, \limsup_{\vep \, \to \, 0^+}\, \frac{\log \log N(X,d,\vep)}{|\log \vep|}.$$

\smallskip

The \emph{topological entropy} is a classical invariant of topological conjugacy, introduced by R. Adler, A. Konheim and M. McAndrew in \cite{AKM}, which informs on the dynamical complexity of systems and sometimes provides a way to select elements from $\cP_f(X)$. It quantifies to what extent nearby orbits diverge as the dynamical system evolves. Recall that it is a non-negative real number or $+\infty$ defined by
$$\mathrm{h_{top}}(f) \, = \, \limsup_{\vep\,\to\, 0^+} \,h_\vep(f)$$
where the topological entropy of $f$ at scale $\vep$ is
\begin{equation}\label{eq:hvep}
h_{\vep}(f) \, = \,\limsup_{n\, \to\, +\infty}\,\frac{1}{n}\,\log S(X,n,\varepsilon)
\end{equation}
and, for each $n \in \N$, $S(X, n,\vep)$ stands for the maximal cardinality of the $(n,\vep)$-separated subsets of $X$ with respect to the Bowen dynamical metric $d_n$, given by
$$d_n(x,y) = \max_{0 \, \leqslant \,j\,\leqslant \,n-1}\,\big\{d(f^j(x), \,f^j(y)\big\}.$$
For example, on a compact metric space with finite box-counting dimension, a Lipschitz map has finite topological entropy. However, if the dynamics is just continuous, the topological entropy may be infinite. Actually, K. Yano proved in \cite{Yano}
that, on compact smooth manifolds with dimension greater than one, the set of homeomorphisms having infinite topological entropy is $C^0$-generic. 

\smallskip

M. Gromov proposed in \cite{Gro99} a new topological invariant of dynamical systems, called mean dimension, which brings together dimension and dynamics. This concept paved the way for the definition of a related metric-dependent entropy-dimensional quantity, proposed by E. Lindenstrauss and B. Weiss in \cite{LW2000}: the \emph{metric mean dimension}, denoted by $\mathrm{mdim}_M\,(X, d, f)$. It may be thought of as a dynamical box-counting dimension, since it is both a dynamical analog of the box-counting dimension and a dimensional analog of the topological entropy. Its upper version is defined by
$$\mathrm{\overline{\mathrm{mdim}}_M}\,(X, d, f)\, = \, \limsup_{\vep \, \to \, 0^+}\, \frac{\limsup_{n \, \to \, +\infty}\,\frac{1}{n}\,\log \,S(X,n,\vep)}{|\log \vep|}.$$
Thus, the upper metric mean dimension vanishes if the topological entropy of $f$ is finite; if, otherwise, $f$ has infinite entropy, it quantifies the speed at which the entropy at scale $\vep$ approaches $+\infty$ as the scale goes to zero. Moreover, it is bounded from above by the upper box-counting dimension of $(X,d)$ (cf. \cite{LW2000}). For example, it is known (cf. \cite{VV}) that, if $X=Y^\N$ is a shift space whose alphabet is a compact metric space $(Y, d_Y)$, the space $X$ is endowed with the metric
$$\rho_Y\,\big((a_k)_{k \, \in \, \N},\,\,(b_k)_{k \, \in \, \N}\big) \,=\, \sum_{k=1}^{+\infty}\, 2^{-k}\,d_Y(a_k, \,b_k)$$
and $\sigma$ is the shift map on $X$, then
\begin{equation}\label{eq:mdimshift}
\mathrm{\overline{\mathrm{mdim}}_M}\,(Y^\N, \rho_Y, \sigma)\, = \, \mathrm{\overline{\dim}_B}(Y,d_Y).
\end{equation}
 
\smallskip

The purpose of this paper is to introduce a counterpart of the metric mean dimension, which we call \emph{dynamical metric order}, based on the metric order instead of the box-counting dimension, and meant to distinguish dynamical systems acting on infinite box-counting dimensional spaces. It is defined by
\begin{equation}\label{def:moint}
\mathrm{\overline{mdim}_{mo}}\,\big(X,d,f\big) = \limsup_{\varepsilon\,\to\, 0^+} \,\frac{\log^+ \Big(\limsup_{n\, \to\, +\infty}\,\frac{1}{n}\,\log S(X,n,\varepsilon)\Big)}{|\log \varepsilon|}
\end{equation}
where $\log^+(t) = \max \{0, \log t\}$, $\log^+ 0 = 0$ and $\log^+ \,(+\infty) = +\infty$. Some of its properties are expected and deduced by methods similar to those used to study the metric mean dimension. For example, the metric order is an upper bound of the dynamical metric order; and, as happens in \eqref{eq:mdimshift} for the metric mean dimension, a full shift satisfies (cf. Theorem~\ref{thm1})
$$\mathrm{\overline{mdim}_{mo}}\,\big(Y^\mathbb{N},
\rho, \sigma\big) \,=\, \mathrm{\overline{mo}}\,\big(Y,d_Y\big).$$
However, there is an important departure in this new approach that is of great value. Given a continuous map $f\colon X \to X$, if $W_1$ stands for a Wasserstein distance on $\cP(X)$ (for which one has $\mathrm{\overline{dim}_B}\,\big(\mathcal{P}(X),W_1\big) = +\infty$), then
$$\mathrm{\overline{mdim}_M}\,\big(\mathcal{P}(X), W_1, f_\ast\big) = 0 \quad \quad \text{or} \quad \quad  \mathrm{\overline{mdim}_M}\,\big(\mathcal{P}(X), W_1, f_\ast\big) = +\infty$$
whether $\mathrm{h_{top}}(f)=0$ or $\mathrm{h_{top}}(f)> 0$, respectively. Whereas, if $0<\mathrm{\overline{dim}_B}\,\big(X,d\big)< +\infty$ and $f$ is Lipschitz with positive topological entropy, then (cf. Corollary~\ref{thm2}, a simple corollary of \cite{BurguetShi})
$$0 \, < \, \mathrm{\overline{mdim}_{mo}}\,\big(\mathcal{P}(X), W_1, f_\ast\big)\,\leqslant\, \mathrm{\overline{dim}_B}(X,d)\,<\,+\infty.$$

\smallskip

We can find in the literature a number of interesting suggestions on how to define a \emph{measure-theoretic dynamical metric order}. See, for example, a comprehensive list in \cite{CYZ}. When considering the metric order, a natural choice seems to be the quantization of measures. More precisely, given $\vep > 0$, the \emph{quantization number} $Q_{\mu,D}(\vep)$ of $\mu \in \mathcal{P}(X)$ at the scale $\vep$ with respect to the metric $D$ in $\mathcal{P}(X)$ is the least positive integer $N$ such that there exists $\nu\in\cP(X)$ supported on a set of cardinality $N$ with $D(\mu, \nu) < \vep$. The \emph{upper quantization order} of $\mu\in\cP(X)$ with respect to the metric $D$ is defined by
\begin{eqnarray*}
\mathrm{\overline{qo}_D}(\mu) &=& \limsup_{\vep \, \to \, 0^+} \,\frac{\log \log Q_{\mu,D}(\vep)}{|\log \vep|}.
\end{eqnarray*}
This notion yields an immediate variational principle for the dynamical metric order of a full shift, with equilibrium states in $\cP(X)$ (cf. Corollary~\ref{cor:vp}). More precisely, if $(Y,d_Y)$ is a compact metric space, $X=Y^\N$ and $D \in \{LP\}\cup\{W_p \colon 1 \leqslant p < +\infty\}$, then 
$$\mathrm{\overline{mdim}_{mo}}\,\big(Y^\mathbb{N},\rho_Y, \sigma\big) \,=\, \max_{\mu \, \in \, \mathcal{P}(Y)}\, \mathrm{\overline{qo}_{D}}(\mu).$$ 
In Section~\ref{se:QDVP}, we discuss a dynamical variant of the quantization of measures that also builds a variational principle of interest, with equilibrium states invariant by the dynamics (see Corollary~\ref{thm4}). This variant was inspired by \cite{BB}, where the authors consider, for each $n \in \N$, dynamical Wasserstein distances $W_{p,n}$, for $1 \leqslant p < +\infty$, and a dynamical L\'evy-Prokhorov distance $LP_n$, defined using the dynamical metric $d_n$ in $X$ instead of the original metric $d$. The proof of such a variational principle is just a slight modification of the argument given in \cite[Theorem~A and B]{CP2026}, where a classical variational principle is established for the metric mean dimension.
\smallskip

The aforementioned dynamical distances on $\cP(X)$ led us to ask how the metric order of $\cP(X)$ varies with $n$. We establish the following lower bound (cf. Theorem~\ref{thm3}):
$$\limsup_{n \, \to \, +\infty}\,
\frac{\mathrm{\overline{mo}}\,\big(\mathcal{P}(X), W_{1,n}\big)}{n} \,=\, \limsup_{n \, \to \, +\infty}\,
\frac{\mathrm{\overline{dim}_B}\,\big(X,d_n\big)}{n} \, \geqslant \, \mathrm{\overline{mdim}_M}\,\big(X,d,f\big).$$
\smallskip

The paper is structured as follows. In the next section, we gather information on the notation and concepts we will use. We thoroughly explain our contributions in Sections~\ref{se:shifts} and \ref{se:growth}, where we also discuss several examples. In Sections~\ref{se:QDVP} and \ref{se:complex}, we describe the two main motivations for our new approach.  Section~\ref{se:examp} gathers more examples that help to clarify the sharpness and scope of our results. We present a number of basic properties of the dynamical metric order in Section~\ref{se:basic}, and the proofs of our contributions appear in Sections~\ref{se:proofA}--\ref{se:proofCC}.

\section{Definitions}\label{sec:definition}
Let $(X,d)$ stand for a compact metric space and $f \colon X \to X$ be a continuous map.

\subsection{ Bowen dynamical metrics}
Given $n\in \mathbb{N}$, the dynamical metric $d_n \colon X \times X \, \to \,[0,+\infty[$ determined by the distance $d$ in $X$ and the map $f$ is defined by 
\begin{displaymath}
d_n(x,y)=\max\,\Big\{d(x,y),\,d(f(x),f(y)),\,\dots,\,d(f^{n-1}(x),f^{n-1}(y))\Big\}.
\end{displaymath}
It is easy to check that $d_n$ is indeed a metric and that it generates the same topology as $d$. Given $\varepsilon > 0$, $n \in \mathbb{N}$ and a point $x \in X$, the open $(n, \varepsilon)$-ball around $x$ is the set
\begin{displaymath}
\mathit{B}_{n}(x, \varepsilon) = \{ y \in X :d_n( x, y) < \varepsilon \}.
\end{displaymath}
    
\subsection{Separated sets, spanning sets and open covers}
Given $\varepsilon > 0$ and $n \in \mathbb{N}$, the next  three quantities count the number of orbit segments of length $n$ that are distinguishable at scale $\vep$.

A set $E \subset X$ is \emph{$(n,\varepsilon)$-separated} if any two distinct points in $E$ are at least $\vep$-apart in the metric $d_n$, that is, $d_n(x,y) > \varepsilon$ for every $x\neq y \in E$. We denote by $S(X, n,\varepsilon)$ the maximal cardinality of the $(n,\varepsilon)$-separated subsets of $X$. When $n=1$, we simplify this notation saying that the set $E$ is $\varepsilon$-separated and changing the notation to $S(X,d,\vep)$. Thereby, for every $\vep>0$ and $n\in\mathbb N$, $S(X, n, \vep) \,=\, S(X, d_n,\vep).$

A set $F\subset X$ is said to be \emph{$(n,\varepsilon)$-spanning} if for any $x\in X$ there exists $z\in F$ such that $d_n(x,z)<\varepsilon$. We denote by $R(X,n,\varepsilon)$ the minimal cardinality of the $(n,\varepsilon)$-spanning subsets of $X$. Again, when $n=1$, we refer to $\varepsilon$-spanning sets, with minimal cardinality denoted $R(X, d,\vep)$. So, 
$R(X, n, \vep) \,=\, R(X, d_n,\vep)$ for every $\vep>0$ and $n\in\N.$

Given $\vep >0$ and $n \in \N$, the \emph{$\vep$-covering number of $(X,d_n)$}, denoted by $\mathcal{N}(X,n,\vep)$, is \emph{the minimum cardinality of a covering of $X$ by sets of $d_n$-diameter less than $\vep$} (the diameter of a set is the supremum of the distances between pair of points in the set). When $n=1$, we refer to $\mathcal{N}(X,d,\vep)$, hence $\mathcal{N}(X, n, \vep) \,=\, \mathcal{N}(X, d_n,\vep)$ for any $\vep>0$ and $n\in\N.$

We observe that, due to the compactness of $X$, the values of $S(X,n,\varepsilon)$, $R(X,n,\varepsilon)$ and $\mathcal{N}(X,n,\vep)$ are finite for every $n\in\mathbb{N}$ and $\vep>0$. In addition, the previous three quantities are related by (cf. \cite[Lemma 2.5.1]{BrinS}:
\begin{equation}\label{eq:srn} 
\mathcal{N}(X,n,2\vep) \, \leqslant \, R(X,n,\vep) \, \leqslant \, S(X,n,\vep)\, \leqslant \, \mathcal{N}(X,n,\vep).    
\end{equation}
Moreover, 
\begin{equation}\label{eq:cov} 
\mathcal{N}(X,n+m,\vep) \, \leqslant \, \mathcal{N}(X,n,\vep)\, . \, \mathcal{N}(X,m,\vep)
\end{equation}
thus the sequence $\big(\log \mathcal{N}(X,n,\vep)\big)_{n\, \in\, \N}$ is sub-additive.

\subsection{ Box-counting dimension}\label{sse:boxdim}
Given $\vep > 0$, let $N(X,d,\vep)$ stand for \emph{the smallest number of open balls in the metric $d$ with radius $\vep$ which cover $X$}. We observe that, for every $\vep >0$, one has
\begin{equation}\label{eq:box}
S(X,d,2\vep) \, \leqslant \, N(X,d,\vep) \, \leqslant\, S(X,d, \vep).
\end{equation}
The \emph{upper box-counting dimension of $(X,d)$} is defined by
$$\mathrm{\overline{dim}_B}(X,d) \,=\, \limsup_{\vep \, \to \, 0^+}\,\frac{\log N(X,d,\vep)}{|\log \vep|}.$$
The lower box-counting dimension $\mathrm{\underline{dim}_B}(X,d)$ is defined by taking $\liminf$ instead of $\limsup$. Notice that, due to \eqref{eq:box}, the upper box-counting dimension may be computed using $(n,\vep)$-separated sets, that is,
$$\mathrm{\overline{dim}_B}(X,d) \,=\, \limsup_{\vep \, \to \, 0^+}\,\frac{\log S(X,d,\vep)}{|\log \vep|}$$
(and similarly for the lower version). We observe that, by \eqref{eq:srn} and \eqref{eq:box}, we also have
$$\mathrm{\overline{dim}_B}(X,d) \,=\, \limsup_{\vep \, \to \, 0^+}\,\frac{\log \mathcal{N}(X,d,\vep)}{|\log \vep|}.$$

\subsection{ Metric mean dimension}\label{sse:mdim}
The \emph{upper} and \emph{lower metric mean dimensions} are labels of complexity specially devised for dynamical systems with infinite topological entropy. They are denoted by $\overline{\mathrm{mdim}}_M(X,d,f)$ and $\underline{\mathrm{mdim}}_M(X,d,f)$, respectively, to emphasize their dependence on the fixed metric $d$ of the space $X$ where the dynamics $f$ acts. They are defined by
\begin{eqnarray*}
\mathrm{\overline{mdim}_M}\big(X,d,f\big) &=& \limsup_{\vep \, \to \, 0^+}\,\frac{\limsup_{n\, \to \, +\infty}\, \frac{1}{n}\,\log S(X,n,\vep)}{|\log \vep|}\\
\mathrm{\underline{mdim}_M}\big(X,d,f\big) &=& \liminf_{\vep \, \to \, 0^+}\,\frac{\limsup_{n\, \to \, +\infty}\,\frac{1}{n}\,\log S(X,n,\vep)}{|\log \vep|}.
\end{eqnarray*}
Due to \eqref{eq:srn}, the metric mean dimensions may be computed using $R(X,n,\vep)$ or $\mathcal{N}(X,n,\vep)$ instead of $S(X,n,\vep)$. It is known (cf. \cite[Section~5]{VV}) that 
\begin{equation}\label{eq:mdimbox}
\mathrm{\overline{mdim}_M}\big(X,d,f\big) \, \leqslant \, \mathrm{\overline{dim}_B}\big(X,d).
\end{equation}

\subsection{ Metric order} The \emph{upper} and \emph{lower metric orders of $(X,d)$}, defined by Kolmogorov and Tikhomirov \cite{Kolmogorov-Tikhomirov}, are given, respectively, by
\begin{align*}
\overline{\mathrm{mo}}\,(X,d) = \limsup_{\varepsilon\,\to\, 0^+}\,\frac{\log\log S(X,d,\vep)}{|\log \varepsilon|}
\qquad\text{and}\qquad
\underline{\mathrm{mo}}\,(X,d) = \liminf_{\varepsilon\,\to\, 0^+}\,\frac{\log\log S(X,d,\varepsilon)}{|\log \varepsilon|}
\end{align*} 
where it is agreed that $\log 0 = 0$. In case both quantities coincide, we simply denote them by ${\mathrm{mo}}\,(X,d)$, which is named \emph{the metric order of the set $X$ with respect to the distance $d$}. By \eqref{eq:box}, in the previous limits we may use $\mathcal{N}(X,d,\vep)$ instead of $S(X,d,\vep)$.
Observe that, since $\log t < t$ for every $t > 0$, one has
$$\overline{\mathrm{mo}}\,(X,d) \, \leqslant \, \overline{\mathrm{dim}_B}\,(X,d).$$

\subsection{Dynamical metric order}\label{sse.mmd}
The \emph{dynamical metric order of $f$ in $X$ with respect to $d$} is defined by 
$$\mathrm{\overline{mdim}_{mo}}\,\big(X, d, f\big) = \limsup_{\varepsilon\,\to\, 0^+} \,\frac{\log^+ \Big(\limsup_{n\, \to\, +\infty}\,\frac{1}{n}\,\log S(X,n,\varepsilon)\Big)}{|\log \varepsilon|}$$
where $\log^+(t) = \max \{0, \log t\}$, $\log 0 = 0$ and $\log \,(+\infty) = +\infty$. Similarly, we define the lower dynamical metric order of $f$ in $X$ with respect to $d$ by 
$$\mathrm{\underline{mdim}_{mo}}\,\big(X,d, f\big) = \liminf_{\varepsilon\,\to\, 0^+} \,\frac{\log^+ \Big(\limsup_{n\, \to\, +\infty}\,\frac{1}{n}\,\log S(X,n,\varepsilon)\Big)}{|\log \varepsilon|}.$$
The relations \eqref{eq:srn} and \eqref{eq:cov} imply that the dynamical metric order can be computed using $R(X,n,\vep)$ or $\mathcal{N}(X,n,\vep)$ instead of $S(X,n,\vep)$. 
\smallskip

From \eqref{eq:mdimbox} and the inequality $\log t < t$ for every real $t>0$, we conclude that
\begin{equation}\label{eq:leq}
\mathrm{\overline{mdim}_{mo}}\,\big(X,d, f\big) \,\leqslant \,\mathrm{\overline{mdim}_{M}}\,\big(X,d, f\big) \, \leqslant \, \mathrm{\overline{dim}_{B}}\,\big(X,d\big)
\end{equation}
and similarly for the lower dynamical metric order. Observe that \eqref{eq:leq} implies that, if either $\mathrm{\overline{dim}_{B}}\,\big(X,d\big) = 0$ or $\mathrm{h_{top}}(f) < +\infty$, then 
$$\mathrm{\overline{mdim}_{mo}}\,\big(X,d,f\big) \,= \,\mathrm{\overline{mdim}_{M}}\,\big(X,d,f\big) \, =  \, 0.$$
Example~\ref{ex:null} shows that the inequalities \eqref{eq:leq} may be strict.

\smallskip

\subsection{$W_p$ and $LP$ metrics on $\mathcal{P}(X)$}\label{sse:metric}
The space $\mathcal{P}(X)$ of the Borel probability measures on $X$ is metrizable if endowed with the weak$^*$-topology. The classic distances on $\cP(X)$ are the \emph{Wasserstein distances} and the \emph{L\'evy-Prokhorov distance}. The former are defined by
$$W_p(\mu,\nu) = \inf_{\pi\,\in\,\Pi(\mu,\nu)}\,\left(\int_{X\times X}\,[d(x,y)]^p\;d\pi(x,y)\right)^{1/p}$$
where $p \in [1, +\infty[$ and $\Pi(\mu,\nu)$ denotes the set of probability measures on the product space $X\times X$ with marginals $\mu$ and $\nu$ (see \cite{Vi} and references therein for more details). The latter is defined by
$$LP(\mu,\nu) = \inf\,\Big\{\varepsilon > 0 \colon \,\, \forall\, E\subset X \quad  \forall\, \,\text{$\varepsilon$-neighborhood $V_\varepsilon(E)$ of $E$ one has} $$
$$\quad \quad \quad \quad \quad \quad \quad \quad \quad \nu(E)\leqslant \mu(V_\varepsilon(E))+\varepsilon \quad \text{ and } \quad \mu(E)\leqslant \nu(V_\varepsilon(E))+\varepsilon \Big\}.$$
We refer the reader to \cite{Bil} for more information on these distances.

Whenever we consider in $(X,d)$ the dynamical metric $d_n$, for some $n \in \N$, then the corresponding  Wasserstein  and L\'evy-Prokhorov distances will be denoted by $LP_n$ and $W_{p,n}$, respectively.

\smallskip

\subsection{Quantization order}\label{sse:Qorder}

The quantization of measures, as addressed in \cite{GraLus}, consists in approximating, at arbitrarily small resolution with respect to a fixed distance $D$ in $\cP(X)$, a given measure by another with finite support. 

\begin{definition}\label{def:qn}
Given $\mu\in\cP(X)$ and $\vep > 0$, the \emph{quantization number $Q_{\mu,D}(\vep)$ of $\mu$ at a scale $\vep$ with respect to the metric $D$} is the least positive integer $N$ such that there exists $\nu\in\cP(X)$ supported on a set of cardinality $N$ with $D(\mu, \nu) < \vep$. 
\end{definition}

The previous definition admits a reformulation when $D$ is the L\'evy-Prokhorov metric.

\begin{lemma}\label{lemma0}
\cite[Proposition 3.3]{BB} The quantization number $Q_{\mu, LP}(\vep)$ for the $LP$ metric is the least number of closed balls with radius $\vep$ that cover any subset of $X$ with $\mu$-measure at least $1-\vep$. Therefore, for every $\vep > 0$ and $\mu \in \cP(X)$ 
\begin{equation}\label{eq:QN}
Q_{\mu,LP}(\vep) \,\leqslant \, \mathcal{N}(X,d,\vep).
\end{equation}
\end{lemma}

The following is a similar characterization of the quantization number when using a Wasserstein metric.

\begin{lemma}\label{lemma01}\cite[Propositions 3.2 and 3.4]{BB} The quantization number $Q_{\mu,W_p}(\vep)$ for a $W_p$ metric, where $1 \leqslant p < +\infty$, is  the minimal cardinality of the finite sets $F  \subset X$ such that
$$\int_X \, \big(d(x, F)\big)^p\, d\mu(x) \,\leqslant\, \vep^p.$$ 
Moreover, for every $\vep > 0$ and $\mu \in \cP(X)$ 
\begin{equation}\label{eq:QNW}
Q_{\mu,W_p}(\vep) \,\leqslant \, \mathcal{N}(X,d,\vep).
\end{equation}
\end{lemma}

\begin{definition}\label{def:qorder}
The \emph{upper} and \emph{lower quantization orders of $\mu\in\cP(X)$ with respect to the metric $D$} are defined by
\begin{eqnarray*}
\mathrm{\overline{qo}_D}(\mu) &=& \limsup_{\vep \, \to \, 0^+} \,\frac{\log \log Q_{\mu,D}(\vep)}{-\log \vep}\\
\mathrm{\underline{qo}_D}(\mu) &=& \liminf_{\vep \, \to \, 0^+} \,\frac{\log \log Q_{\mu,D}(\vep)}{-\log \vep}.
\end{eqnarray*}
\end{definition}

We observe that, by \eqref{eq:QN} and \eqref{eq:QNW}, for every $\mu\in\cP(X)$ one has
$$\mathrm{\underline{qo}_{D}}(\mu) \, \leqslant \, \underline{\mathrm{mo}}\,(X,d)   \quad \quad \text{ and } \quad \quad \, \mathrm{\overline{qo}_{D}}(\mu) \, \leqslant \, \overline{\mathrm{mo}}\,(X,d).$$


\section{Shifts on infinite-dimensional alphabets}\label{se:shifts}

In the space $X^\mathbb{N}$ of the unilateral sequences generated by the alphabet $(X,d)$, endowed with the product topology, we consider the following compatible metric: given $\underline{x} = (x_k)_{k \, \in \, \N}$ and $\underline{y} = (y_k)_{k \, \in \, \N}$, define
\begin{equation}\label{de:rho}
\rho_{_X}(\underline{x},\underline{y}) \,=\, \sum_{k=1}^{+\infty}\, \frac{d(x_k, \,y_k)}{2^{k}}.
\end{equation}
Let $\sigma\colon X^\mathbb{N} \to X^\N$ be the shift, defined by $\sigma\big((x_k)_{k \, \in \, \N}\big) = (x_{k+1})_{k \, \in \, \N}$. It is known (cf.~\cite{VV}) that 
\begin{align}\label{eq:mdim}
\mathrm{\overline{mdim}_M}\,\big(X^\mathbb{N},\rho_{_X},\sigma\big)=\mathrm{\overline{dim}_B}\,\big(X,d\big).
\end{align}
The next result shows that, when $\mathrm{\overline{dim}_B}\,\big(X,d\big) = +\infty$, the equality \eqref{eq:mdim} has a counterpart where the role of the box-counting dimension is played by the metric order.

\begin{maintheorem}\label{thm1}
Let $(X,d)$ be a compact metric space. Then
$$\mathrm{\overline{mdim}_{mo}}\,\big(X^\mathbb{N},\rho_{_X}, \sigma\big) \,=\, \mathrm{\overline{mo}}\,\big(X,d\big)$$
and, for every continuous map  $f\colon X\to X$,
$$\mathrm{\overline{mdim}_{mo}}\,\big(X,d,f\big)\,\leqslant \,\mathrm{\overline{mo}}\,\big(X,d\big).$$
\end{maintheorem}

Since the upper metric order has the intermediate property (cf. \cite[Theorem 4]{CRV4}), the previous result implies that, within the shift maps, the dynamical metric order has this property as well. More precisely, given a compact metric space $(X, d)$ and $\beta$ that satisfies $0 \leqslant \beta \leqslant \mathrm{\overline{mo}}\,\big(X,d\big)$, there exists a subset $Y_\beta \subset X$ such that $\mathrm{\overline{mo}}\,\big(Y_\beta,d\big) = \beta$. Therefore, given 
$\alpha$ such that $0 \leqslant \alpha \leqslant \mathrm{\overline{mdim}_{mo}}\,\big(X^\mathbb{N},\rho_{_X}, \sigma\big)$, by Theorem~\ref{thm1} one has 
$$0 \leqslant \alpha \leqslant \mathrm{\overline{mo}}\,\,\big(X,d\big)$$
so there exists a subshift $Y_\alpha^\N \subset X^\N$ such that $\mathrm{\overline{mdim}_{mo}}\,\big(Y_\alpha^\mathbb{N},\sigma,\rho_{_{Y_\alpha}}\big) = \mathrm{\overline{mo}}\,\big(Y_\alpha,d\big) =  \alpha$.
\smallskip

\begin{example}\label{ex:null}
\emph{Let $(X,d)$ be a compact metric space such that $0 < \mathrm{\overline{dim}_B}\,\big(X,d\big) < +\infty$.
For instance, $X = [0,1]$ or $X= \{0\}\cup\{1/k \colon k \in \N\}$, whose box-counting dimensions are $1$ and $1/2$, respectively. Then, by \eqref{eq:mdim}, one has $\mathrm{\overline{mdim}_M}\,\big(X^\mathbb{N},\rho_{_X},\sigma\big)=\mathrm{\overline{dim}_B}\,\big(X,d\big) > 0.$}
 
\begin{lemma}\label{le:mo1} If $\mathrm{\overline{dim}_B}\,\big(X,d\big) < +\infty$, then $\mathrm{mo}\big(X,d\big) = 0$.
\end{lemma}

\begin{proof}
Given $\delta>0$, take  $\vep_0 = \vep_0(\delta) >0$ such that
\begin{align*}
 0 < \vep < \varepsilon_0 \quad \Rightarrow \quad S(X,d,\varepsilon)< e^{\big(\mathrm{\overline{dim}_B(X,d)}\,+\,\delta\big)\,|\log\vep|}.
\end{align*}
Thus, 
\begin{align*}
  \frac{\log\log S(X,d,\varepsilon)}{|\log\varepsilon|}\,<\,\,\frac{\log \big(\mathrm{\overline{dim}_B(X,d)}\,+\,\delta\big)}{|\log\varepsilon|} + \frac{\log|\log\varepsilon|}{|\log\varepsilon|} \quad\quad \forall\, \vep\, \in\, \,\,]0,\vep_0[
\end{align*}
hence
\[\mathrm{\overline{mo}}\,\big(X,d\big)\,\leqslant\,\limsup_{\varepsilon\,\to\, 0^+}\left(\frac{\log \big(\mathrm{\overline{dim}_B(X,d)}\,+\,\delta\big)}{|\log\varepsilon|} + \frac{\log|\log\varepsilon|}{|\log\varepsilon|}\right)\,=\,0.\]
Thus, $\mathrm{mo}\big(X,d\big) = 0.$
\end{proof}

\emph{Since we are assuming that $\mathrm{\overline{dim}_B}\,\big(X,d\big) < +\infty$, by Theorem~\ref{thm1} and Lemma~\ref{le:mo1}, one has
\[\mathrm{\overline{mdim}_{mo}}\,
\big(X^\N, \rho_{_{X}}, \sigma\big) \,=\, \mathrm{\overline{mo}}\,\big(X,d\big) \,=\, 0. \quad \blacksquare\]}
\end{example} 

We note that Lemma~\ref{le:mo1} implies that, if $\mathrm{\overline{dim}_B}\,\big(X,d\big) < +\infty$, then the dynamical metric order of $(X,d,f)$ is zero for any continuous map $f \colon X \to X$.

\smallskip

The previous example makes plain that, for alphabets $X$ with finite box-counting dimension, the metric mean dimension is a sharper label of complexity for the shift map on $X^\N$ than the dynamical metric order. This is no longer true when the box-counting dimension of $X$ is infinite, as the next example illustrates.

\begin{example}\label{ex:push}
\emph{Let $(Y,d_Y)$ be a compact metric space with $0<\mathrm{dim_B}\,\big(Y,d_Y\big) < +\infty$ and let $X = \mathcal{P}(Y)$ endowed with a metric $D \in \{W_p \colon 1 \leqslant p < +\infty\}\cup \{LP\}$.}

\begin{lemma}\label{le:dim1} 
$\mathrm{mo}\big(\mathcal{P}(Y),D\big) = \mathrm{dim_B}\,\big(Y,d_Y\big)$ and $\mathrm{\overline{dim}_B}\,\big(\mathcal{P}(Y),D\big) = +\infty$.
\end{lemma}

\begin{proof}
Regarding the first equality, recall from \cite[Theorem~1.3]{BB} that, for any compact metric space $(Z,d_Z)$, 
\begin{equation}\label{eq:dimmo}
 \mathrm{\underline{dim}_B}(Z,d_Z)\,\leqslant\, \mathrm{\underline{mo}}\,(\mathcal{P}(Z), D) \,\leqslant\,\mathrm{\overline{mo}}\,(\mathcal{P}(Z), D) \,\leqslant\, \mathrm{\overline{dim}_B}(Z,d_Z).
\end{equation}
Since we are assuming that $\mathrm{\underline{dim}_B}(Y,d_Y) = \mathrm{\overline{dim}_B}(Y,d_Y)$, then \eqref{eq:dimmo} yields
\begin{equation}\label{eq:md}
\mathrm{mo}\big(\mathcal{P}(Y),D\big) \,=\, \mathrm{dim}_B\,\big(Y,d_Y\big).
\end{equation}
The equality $\mathrm{\overline{dim}_B}\,\big(\mathcal{P}(Y),D\big) = +\infty$ is a direct consequence of Lemma~\ref{le:mo1}, the equality \eqref{eq:md} and the hypothesis $\mathrm{dim_B} \,\big(Y,d_Y\big) > 0$.
\end{proof}

\emph{From \eqref{eq:mdim} and Lemma~\ref{le:dim1}, we deduce that
$$\mathrm{\overline{mdim}_M}\,\big(\mathcal{P}(Y)^\mathbb{N}, \rho_{_{\cP(Y)}}, \sigma\big) \, = \,\mathrm{\overline{dim}_B}\,\big(\mathcal{P}(Y),D\big) = +\infty.$$
On the other hand, by Theorem~\ref{thm1} and \eqref{eq:md},
\[0 \,<\, \mathrm{\overline{mdim}_{mo}}\,\big(\mathcal{P}(Y)^\mathbb{N}, \sigma, \rho_{_{\cP(Y)}}\big) \, = \,  \mathrm{mo}\big(\mathcal{P}(Y), D\big) \,=\,  \mathrm{dim_B}\,\big(Y,d_Y\big) \,<\, +\infty. \quad \blacksquare\]}
\end{example} 

\medskip

The previous example also shows that the dynamical metric order can span any non-negative value: for every $0 \leqslant \beta < +\infty$, there is a compact metric space $(Y,d)$ whose box-counting dimension is $\beta$ (cf. \cite{Feng-Zhiying}); hence, $\mathrm{\overline{mdim}_{mo}}\,
\big(\mathcal{P}(Y)^\N, \rho_{_{\cP(Y)}},\sigma\big) = \beta.$

\medskip

The upper quantization order of $\mu\in\cP(X)$, defined in Subsection~\ref{sse:Qorder}, is related to the metric order. More precisely, 
\begin{itemize}
\item \cite[Proposition 3.4]{BB} \emph{For every $\mu \in \cP(X)$, one has $\,\,\mathrm{\overline{qo}_D}(\mu) \,\leqslant \,\mathrm{\overline{mo}}\,\big(X,d\big)$.}
\smallskip
\item  \cite[Theorem 3.9]{BB} \emph{There is $\mu_X \in \cP(X)$ such that $\mathrm{\overline{qo}_D}(\mu_X) \,=\,\mathrm{\overline{mo}}\,\big(X,d\big)$.}
\end{itemize}
\medskip

\noindent The next variational principle is an easy consequence of this information and Theorem~\ref{thm1}.

\begin{corollary}\label{cor:vp}
Let $(X,d)$ be a compact metric space and $D \in \{LP\}\cup\{W_p \colon 1 \leqslant p < +\infty\}$. Then
$$\mathrm{\overline{mdim}_{mo}}\,\big(X^\mathbb{N},\rho_{_X}, \sigma\big) \,=\, \max_{\mu \, \in \, \mathcal{P}(X)}\, \mathrm{\overline{qo}_{D}}(\mu).$$ 
\end{corollary}

\section{Dynamical quantization order}\label{se:QDVP}

Recall from Section~\ref{sse:metric} that, given $n \in \N$, we denote by $LP_n$ and $W_{p,n}$ the Wasserstein and L\'evy-Prokhorov distances in $\cP(X)$, respectively, induced by the metric $d_n$ in $X$.
\smallskip

\begin{definition}
Consider $D \in \{LP\}\cup\{W_{p} \colon 1 \leqslant p < +\infty\}$ and, given $n \in \N$, the corresponding dynamical metric $D_n \in \{LP_n\}\cup\{W_{p,n} \colon 1 \leqslant p < +\infty\}$. For $\mu\in\cP(X)$ and $\vep > 0$, the \emph{dynamical quantization number $Q_{\mu,D_n}(\vep)$ of $\mu$ at scale $\vep$} is the least positive integer $K=K(n,\vep,\mu)$ such that there exists $\nu\in\cP(X)$ supported on a set of cardinality $K$ with $D_n(\mu, \nu) < \vep$. The \textit{upper and lower dynamical quantization order} of $\mu $ with respect to the metric $D$ are defined, respectively, by 
\begin{align}\label{dynamical-quantization}
\mathrm{\overline{mdim}_{mo}}\big(X,d,f,D,\mu\big)
&=\,\limsup_{\vep \,\to\,0^+}\,\frac{\log^+\left(\limsup_{n\, \to\, +\infty}\,\frac1n\,\log Q_{\mu,D_n}(\vep)\right)}{|\log\varepsilon|}\\ \nonumber 
\mathrm{\underline{mdim}_{mo}}\big(X,d,f,D,\mu \big)
&=\,\liminf_{\vep\,\to\,0^+}\,\frac{\log^+\left(\limsup_{n\,\to\,+\infty}\,\frac1n\,\log Q_{\mu,D_n}(\vep)\right)}{|\log\varepsilon|}.
\end{align}
\end{definition}

When $f$ is a homeomorphism, the dynamical metric order satisfies a classical variational principle with respect to the upper dynamical quantization order. This property is a consequence of \cite[Theorems~A and B]{CP2026}, whose proofs admit a direct adaptation to the dynamical metric order.

\begin{corollary}\label{thm4}
Let $(X,d)$ be a compact metric space, $f\colon X\to X$ be a homeomorphism and $D$ be a metric in $\{LP\}\cup\{W_{p} \colon 1 \leqslant p < +\infty\}$. Then
\[\mathrm{\overline{mdim}_{mo}}\,\big(X,d,f\big)\,\,=\,
\max_{\mu\,\in\,\cP_f(X)}\,\mathrm{\overline{mdim}_{mo}}\,\big(X,d,f,D,\mu\big). \]
\end{corollary}

\section{Complexity of the induced map}\label{se:complex}

Let $(X,d)$ be a compact metric space, $f\colon X \to X$ be a continuous map, $\mathcal{P}(X)$ be the space of Borel probability measures on $X$ endowed with the weak$^*$-topology and $f_\ast \colon \mathcal{P}(X) \to \mathcal{P}(X)$ be the map induced by $f$ on $\mathcal{P}(X)$. It is known (cf. \cite{BurguetShi}) that, if $D \in \{W_p \colon 1 \leqslant p < +\infty\}\cup \{LP\}$, then
\begin{equation}\label{eq:top}
\mathrm{h_{top}}(f) > 0\quad \Leftrightarrow \quad \mathrm{\overline{mdim}_M}\,\big(\mathcal{P}(X),D,f_\ast\big) \,=\, +\infty.
\end{equation}
In addition, by \cite{BSigmund,GlasnerWeiss} the following dichotomy holds:
\begin{align}\label{eq:entropy-eq}
\mathrm{h_{top}}(f) \,>\, 0
&\quad \Leftrightarrow \quad \mathrm{h_{top}}(f_\ast)=+\infty\\\nonumber
\mathrm{h_{top}} (f)\,=\, 0 
&\quad \Leftrightarrow \quad \mathrm{h_{top}}(f_\ast) \,=\,0.
\end{align}
Besides, the definitions yield that 
\begin{equation}\label{eq:hm}
\mathrm{h_{top}}(f_\ast) \,=\, 0 \quad \Rightarrow \quad \mathrm{\overline{mdim}_M}\,\big(\mathcal{P}(X),D,f_\ast\big)\,=\,0.
\end{equation}
The equivalence relations \eqref{eq:top} and \eqref{eq:entropy-eq} together with \eqref{eq:hm} imply that the  metric mean dimension of $f_\ast$ only admits two values, namely zero and infinity. In contrast, the dynamical metric order of $f_\ast$ is aware of both the topological entropy of $f$ and the dimension of $X$, and we will show that, under mild assumptions, it exhibits a wider range of values.

\smallskip

The next application of Theorem~\ref{thm1} is an easy consequence of \eqref{eq:dimmo}.

\begin{corollary}\label{cor:push}
Let $(X,d)$ be a compact metric space and $f\colon X\to X$ be a continuous map. If $D \in \{W_p \colon 1 \leqslant p < +\infty\}\cup \{LP\}$, then
$$\mathrm{\overline{mdim}_{mo}}\,\big(\mathcal{P}(X),D,f_\ast\big) \,\leqslant \,\mathrm{\overline{dim}_B}(X,d).$$
\end{corollary}

We observe that, if $\mathrm{\overline{dim}_B}(X,d) = 0$ then, by Corollary~\ref{cor:push}, $\mathrm{\overline{mdim}_{mo}}\,\big(\mathcal{P}(X),D,f_\ast\big) = 0$ for every continuous map $f\colon X\to X$. Assume, otherwise, that $0<\mathrm{\overline{dim}_B}\,\big(X,d\big)< +\infty$ and $f\colon X\to X$ has positive topological entropy; then, by \eqref{eq:top},  
$\mathrm{\overline{mdim}_M}\,\big(\mathcal{P}(X),D,f_\ast\big) = +\infty$, whereas Corollary~\ref{cor:push} ensures that 
$$\mathrm{\overline{mdim}_{mo}}\,\big(\mathcal{P}(X),D,f_\ast\big)\,<\,+\infty.$$
In the next subsection, we discuss a class of functions for which the dynamical metric order of the corresponding induced maps is finite and non-zero.

\subsection{A particular case: Lipschitz maps}

Given $\lambda > 0$, a map $f\colon X\to X$ on a compact metric space $(X,d)$ is $\lambda$-Lipschitz if
$$d\big(f(x), \,f(y)\big) \, \leqslant \, \lambda \, d(x,y)\quad \quad \forall\, x, y \in X$$
and 
$$\lambda \, = \, \min \big\{L \in \,\,]0, +\infty[\colon \, \,d\big(f(x), \,f(y)\big)  \leqslant  L \, d(x,y)\quad \forall\, x, y \in X \big\}.$$ 
\smallskip

\noindent It is known (cf. \cite[Theorem 3.2.9]{KH}) that, when $f\colon X \to X$ is $\lambda$-Lipschitz, then its topological entropy is bounded from above by
\begin{equation}\label{eq:Lip}
\mathrm{h_{top}}(f) \, \leqslant\, \max\,\{0, \, \log \lambda\}\, .\,\mathrm{\overline{dim}_{B}}\,\big(X, d\big).
\end{equation}
Therefore, if $f\colon X \to X$ is $\lambda$-Lipschitz and has a positive topological entropy, then $\lambda > 1$ and the upper box-counting dimension of $(X,d)$ is positive. 
\smallskip

The next property of the dynamical metric order is an immediate consequence of Corollary~\ref{cor:push} and \cite[Corollary 19]{BurguetShi}.

\begin{corollary}\label{thm2}
Let $(X,d)$ be a compact metric space with finite upper box-counting dimension and consider a continuous map $f\colon X\to X$. If $f$ is $\lambda$-Lipschitz and has positive topological entropy, then 
\[0 \,<\,\mathrm{\overline{mdim}_{mo}}\,\big(\mathcal{P}(X),W_1,f_\ast\big)\, < \, +\infty.\]
\end{corollary}

In this setting, one always has $\mathrm{\overline{mdim}_{M}}\,\big(\mathcal{P}(X),W_1,f_\ast\big) = +\infty$; therefore, the dynamical metric order provides a better informed estimate of the complexity of $f_\ast$.

\begin{example}
\emph{Consider $X =\{0,1\}^\N$ with the metric  
$$d(\underline{x},\underline{y}) = 
\left\{\begin{array}{lr}
 e^{-k},  \text{ where $k = \min\{j \in \N\colon x_j \neq y_j\}$,} & \quad \text{if this set is nonempty}\\
 0 & \quad \text{ otherwise}
 \end{array}
 \right.$$ and the shift map $\sigma \colon X\to X$. Then $\sigma$ is Lipschitz, has positive topological entropy and $\mathrm{\overline{dim}_B}\,\big(X,d\big) = \log 2$. Indeed:}
 \medskip

\noindent \emph{$(i)$ Given $(x_k)_{k\,\in\,\N}), (y_k)_{k\,\in\,\N} \in X$,
$$d\big(\sigma((x_k)_{k\,\in\,\N}),\,\sigma((y_k)_{k\,\in\,\N})\big) \,=\,d\big((x_{k+1})_{k\,\in\,\N}, \,(y_{k+1})_{k\,\in\,\N})\big) 
\,\leqslant\, e\,d\big((x_k)_{k\,\in\,\N}, \,(y_k)_{k\,\in\,\N}\big).$$}

\noindent \emph{$(ii)$ It is known (cf. \cite{Wa}) that $h_{top}(\sigma) = \log 2$.}
\medskip

\noindent \emph{$(iii)$ 
Given $m\in\mathbb N$, let $\varepsilon=e^{-m}$. Then $S(X,d,\varepsilon)=2^m$, and so $\mathrm{\overline{dim}_B}\,\big(X,d\big)=\log 2$.}
\medskip

\noindent \emph{Therefore, from item $(ii)$ and \eqref{eq:top} we get 
$$\mathrm{\overline{mdim}_{M}}\,\big(\mathcal{P}(X),W_1, \sigma_\ast\big) \,=\, +\infty.$$
On the other hand, by Corollary~\ref{cor:push},  
\[\mathrm{\overline{mdim}_{mo}}\,\big(\mathcal{P}(X),W_1,\sigma_\ast\big) \,\leqslant\, \mathrm{\overline{dim}_B}\,\big(X,d\big) \,=\, \log 2;\]
and, by properties $(i)-(iii)$, Corollary~\ref{thm2} ensures that 
$ \mathrm{\overline{mdim}_{mo}}\,\big(\mathcal{P}(X),W_1,\sigma_\ast\big) > 0. \quad \blacksquare$}
\end{example}

\section{Growth rate of the metric order}\label{se:growth}

The topological emergence of $f$, introduced in \cite{Berger}, evaluates the complexity of a map by the size of the space of its invariant ergodic probability measures. It is defined by
$$Emer_{\mathrm{top}}(f) \, = \,\limsup_{\vep \, \to \, 0^+}\,\frac{\log \log \,\mathcal{E}_{\mathrm{top}}(f)(\vep)}{|\log \vep|}$$
where $\mathcal{E}_{\mathrm{top}}(f)(\vep)$ is the minimal number of $\vep$-open balls of $\mathcal{P}(X)$ in the distance $W_p$ (respectively, $LP$) whose union covers $\mathcal{E}_f(X).$ According to \cite[Section~2]{BB}, for every $1 \leqslant p < +\infty$, 
$$Emer_{\mathrm{top}}(f) \,=\, \mathrm{\overline{mo}}\,\big(\mathcal{E}_f(X), W_p\big)\,=\, \mathrm{\overline{mo}}\,\big(\mathcal{E}_f(X), LP\big).$$

An analogous measure of complexity might be achieved by estimating the size of the space $\cP(X)$ endowed with dynamical metrics. The following class of metrics on $\cP(X)$ was introduced in \cite{BB} and are determined by the Bowen dynamical metrics generated by $f$ on $(X,d)$. More precisely, for each $p\in\,[1,+\infty[$ and $n\in\mathbb{N}$, define
\begin{align}
   W_{p,n}(\mu,\nu) &= \inf_{\pi\,\in\,\Pi(\mu,\nu)}\,\left(\int_{X\times X}\,[d_n(x,y)]^p\;d\pi(x,y)\right)^{1/p}\\
\smallskip & \nonumber \\
LP_n(\mu,\nu) &= \inf\,\mathcal{L}(\mu, \nu, n) \nonumber
\end{align}
where
$$\mathcal{L}(\mu, \nu, n) \, = \, \Big\{\varepsilon > 0 \colon \,\, \forall\, E\subset X \quad  \forall\, \,\text{$\varepsilon$-neighborhood $V_{\varepsilon, \, n}(E)$ of $E$ in the metric $d_n$}$$
$$\quad \quad \quad \quad \quad \quad \quad \quad \nu(E)\,\leqslant\, \mu(V_{\varepsilon, \, n}(E))+\varepsilon \quad \text{ and } \quad \mu(E)\,\leqslant\, \nu(V_{\varepsilon, \, n}(E))+\varepsilon\, \Big\}.$$
Since the notions of metric order and box-counting dimension are metric-dependent, we may wonder how they change when the distance $d$ is replaced by $d_n$. This question is the motivation for the next result.

\begin{maintheorem}\label{thm3}
Let $(X,d)$ be a compact metric space and $f\colon X\to X$ a continuous map. Then
$$\limsup_{n \, \to \, +\infty}\,
\frac{\mathrm{\overline{mo}}\,\big(\mathcal{P}(X), W_{1,n}\big)}{n} \,=\, \limsup_{n \, \to \, +\infty}\,
\frac{\mathrm{\overline{dim}_B}\,\big(X,d_n\big)}{n} \, \geqslant \, \mathrm{\overline{mdim}_M}\,\big(X,d,f\big).$$
\end{maintheorem}

\medskip

Observe that 
$\limsup_{n \, \to \, +\infty}\,\mathrm{\overline{dim}_B}\,\big(X,d_n\big)/n$
is defined like  $\mathrm{\overline{mdim}_M}\,\big(X,d,f\big)$ but with the order of the limits in $\vep$ and $n$ exchanged. In what follows, we refer to it as the \emph{mean box-counting dimension of $(X,d,f)$}, as done in \cite{GutmanSp} with respect to a similar concept used to reformulate the notion of metric mean dimension for subshifts.

\begin{example}
    
\emph{From Theorem~\ref{thm3} and  \eqref{eq:mdim} we deduce that, if $(Y,d_Y)$ is a compact metric space and $\sigma$ is the full shift on $(Y^\N, \rho_Y)$,
then
$$\limsup_{n \, \to \, +\infty}\,
\frac{\mathrm{\overline{dim}_B}\,\big(Y^\N,\,(\rho_Y)_n\big)}{n} \, \geqslant \, \mathrm{\overline{dim}_B}\,\big(Y,d_Y\big).$$
However, this inequality may be strict (see Example~\ref{ex:shift}). An analogous inequality is known in a rather different setting. More precisely, in \cite{CRV3}, the authors showed that, for any $C^0$-generic homeomorphism $f\colon X \to X$ of a compact smooth boundaryless manifold $X$ with topological dimension $\mathrm{dim} X > 1$ endowed with a metric $d$ compatible with its smooth structure, the upper metric mean dimension of $f$ on $(X,d)$ attains its maximum value, which is $\mathrm{dim} X.$ Therefore, from Theorem~\ref{thm3} we conclude that, for $C^0$-generic homeomorphisms $f\colon X \to X$, one has 
$$\limsup_{n \, \to \, +\infty}\,
\frac{\mathrm{\overline{dim}_B}\,\big(X,d_n\big)}{n} \, \geqslant \,\mathrm{dim} X.$$}
\end{example}

\noindent In contrast, in the next example the mean box-counting dimension is strictly smaller than the upper box-counting dimension of the phase space.

\begin{example}
\emph{Let $(X,d)$ be a compact metric space such that $0 < \mathrm{\overline{dim}_B}\,\big(X,d\big) < +\infty$ and $f\colon X\to X$ be a $\lambda$-Lipschitz map. Then $f$ has finite topological entropy and hence zero metric mean dimension.} 
\smallskip

\noindent \emph{$(a)$ If $0 \leqslant \lambda \leqslant 1$ then, for every $n \in \N$ and any pair $x,y \in X$, one has
$$d_n(x,y) \, = \, \max \big\{d(x,y), \cdots, d(f^{n-1}(x), f^{n-1}(y)) \big\} \, = \, d(x,y).$$
Thus, 
$$ \limsup_{n \, \to \, +\infty}\,
\frac{\mathrm{\overline{dim}_B}\,\big(X,d_n\big)}{n} \, = \, \limsup_{n \, \to \, +\infty}\,
\frac{\mathrm{\overline{dim}_B}\,\big(X,d\big)}{n}  \, = \, 0 \, = \, \mathrm{\overline{mdim}_M}\,\big(X,d,f\big) 
\, < \,\mathrm{\overline{dim}_B}\,\big(X,d\big).$$}

\noindent \emph{$(b)$ If $\lambda > 1$ then, for every $n \in \N$ and any pair $x,y \in X$, one has
$$d_n(x,y) \, = \, \max \big\{d(x,y), \cdots, d(f^{n-1}(x), f^{n-1}(y)) \big\} \, \leqslant \, \lambda^n \, d(x,y)$$
and so, for all $\vep > 0$ and $n \in \N$,
$$d_n(x,y) \, \geqslant \, \vep  \quad \Rightarrow \quad d(x,y) \, \geqslant \,\vep / \lambda^n.$$
Therefore,
$$S(X,d_n,\vep) \,\leqslant\, S(X, d,\vep / \lambda^n)$$
and 
\begin{eqnarray*} 
 \limsup_{n \, \to \, +\infty}\,
\frac{\mathrm{\overline{dim}_B}\,\big(X,d_n\big)}{n} & = & \limsup_{n \, \to \, +\infty}\,
\frac{\limsup_{\vep \, \to \, 0^+} \, \frac{\log S\big(X,\,d_n, \,\vep\big)}{|\log \vep|} }{n} \\
& \leqslant & \limsup_{n \, \to \, +\infty}\,
\frac{\limsup_{\vep \, \to \, 0^+} \, \frac{\log S\big(X, \,d, \,\vep / \lambda^n\big)}{|\log \,\vep|}}{n} \\
& = & \limsup_{n \, \to \, +\infty}\,
\frac{\limsup_{\vep \, \to \, 0^+} \, \Big(\frac{\log S\big(X, \,d, \,\vep / \lambda^n\big)}{|\log \,(\vep/\lambda^n)|}\, . \, \frac{|\log \,(\vep/\lambda^n)|}{|\log \vep|}\Big)}{n} \\
& = & \limsup_{n \, \to \, +\infty}\,
\frac{\limsup_{\vep \, \to \, 0^+} \, \frac{\log S\big(X, \,d, \,\vep / \lambda^n\big)}{|\log \,(\vep/\lambda^n)|}}{n} \\
& = & \limsup_{n \, \to \, +\infty}\,
\frac{\mathrm{\overline{dim}_B}\,\big(X,d\big)}{n} \, = \, 0.
\end{eqnarray*}
So,
$$ \limsup_{n \, \to \, +\infty}\,
\frac{\mathrm{\overline{dim}_B}\,\big(X,d_n\big)}{n} \,= \, 0 \,=\, \mathrm{\overline{mdim}_M}\,\big(X,d,f\big) \, < \,\mathrm{\overline{dim}_B}\,\big(X,d\big). \quad \blacksquare$$}
\end{example}

\section{More examples}\label{se:examp}

Here, we address some particularities of the previous concepts and theorems.

\begin{example}\textbf{\emph{The dynamical metric order is metric-dependent.}}\label{ex:Banach}
\medskip

\noindent \emph{Consider the Banach space 
$$\mathcal{B}\,=\,\Big\{(x_n)_{n\,\in\,\mathbb N}\in \mathbb R^{\mathbb N} \colon \,\limsup_{n\,\to\,+\infty}\,|x_n|/n = 0\Big\}$$
endowed with the norm given by $\|(x_n)_{n\,\in\,\mathbb N}\| \,=\,\sup_{n\, \in \, \N} \,|x_n|/n.$
Let $K$ be the compact set
$$K \,=\,\big\{(x_n)_{n\,\in\,\mathbb N}\,\in \mathcal{B}\colon  \,|x_n|\leqslant 1 \quad \forall n \in \mathbb N\} \,=\, [-1,1]^\mathbb{N}$$
and denote by $d$ the distance induced on $K$ by the norm $\|\cdot\|$.}
\smallskip

\emph{Given $0 < \varepsilon < 1$, let $n_0=n_0(\varepsilon)\in\mathbb{N}$ such that 
\[\frac1{n_0+1} \,<\, \frac{\varepsilon}2 \,\leqslant \, \frac1{n_0}.\]
Define $m = \lfloor\frac{1}{\varepsilon}\rfloor$ and the sets 
$$\mathcal{I} \, = \, 
\Big\{0,\,\pm\,\varepsilon/2,\,\pm\,2\varepsilon/2  ,\,\dots,\,\pm \,m\varepsilon/2\Big\} \, \subset \, [-1,1]$$
and
$$L \,=\, \Big\{(x_n)_{n\,\in\,\mathbb N}\in K\colon \,\, x_i \in \mathcal{I} \quad \forall \,1 \leqslant i \leqslant n_0 \,\text{ and } \,x_i = 1 \quad \forall\, i > n_0\Big\}.$$}
\smallskip

\noindent \emph{\textbf{Claim:}} $L$ is an $\varepsilon$-spanning subset of $K$ with cardinality $\left\lfloor 2/\varepsilon \right\rfloor^{n_0}$, and so $\mathrm{\overline{mo}}\,\big(K,d\big)\leqslant 1$.

\medskip

\begin{proof} Clearly, by construction, $L$ has cardinality $\left\lfloor 2/\varepsilon \right\rfloor^{n_0}$. Given $(y_n)_{n\,\in\,\mathbb N} \in K$, take $(x_n)_{n\,\in\,\mathbb N}$ in $L$ such that 
\[|x_n-y_n|\,\leqslant \, \varepsilon/2 \quad \quad \forall \, 1 \leqslant n \leqslant n_0.\]
Then,
\begin{align*}
\|(x_n)_{n\,\in\,\mathbb N} - (y_n)_{n\,\in\,\mathbb N}\| & = \,\max\left\{\sup_{n\,>\,n_0\,+\,1}\,|x_n-y_n|/n,\,\,\max_{1\,\leqslant \, n\,\leqslant \,  n_0}\,|x_n-y_n|/n\right\}\\
&\leqslant \, 
\max\Big\{\varepsilon,\,\varepsilon/2\Big\} \,=\,\varepsilon.
\end{align*}
Thus, $L$ is $\vep$-spanning. Therefore,  
$$\mathrm{\overline{mo}}\,\big(K,d\big)\,\leqslant \, \limsup_{\varepsilon\,\to\,0^+}\,\frac{\log n_0(\varepsilon)+\log\log\frac2\varepsilon}{-\log\varepsilon}\,\leqslant \, \limsup_{\varepsilon\,\to\,0^+}\,\frac{\log \frac2\varepsilon + \log\log\frac2\varepsilon}{-\log\varepsilon}\,=\,1.$$
\end{proof}

\emph{On the other hand,} \\

\noindent \emph{\textbf{Claim:}} $L$ is an $\varepsilon^2/4$-spanning subset of $K$, and so $\mathrm{\overline{mo}}\,\big(K,d\big)\geqslant 1/2$.
\medskip

\begin{proof} Given $(x_n)_{n\,\in\,\mathbb N}\not= (y_n)_{n\,\in\,\mathbb N}$ in $L$, one has
$$\|(x_n)_{n\,\in\,\mathbb N}-(y_n)_{n\,\in\,\mathbb N}\|\,=\,\sup_{n\, \in \,\N}\,|x_n-y_n|/n \,=\,\max_{1\,\leqslant\, n\,\leqslant\, n_0}\,|x_n-y_n|/n \,\geqslant \,  \frac{\varepsilon}{2n_0}\,\geqslant \,\frac{\varepsilon^2}4.
$$
So, $L$ is $\vep$-separated. Hence,
\begin{align*}
\mathrm{\overline{mo}}\,\big(K,d\big)
    &\geqslant \,\limsup_{\varepsilon\,\to\,0^+}\,\frac{\log n_0(\varepsilon)+\log\log\frac4{\varepsilon^2}}{\log\frac4{\varepsilon^2}}\\
    &\geqslant \, \limsup_{\varepsilon\,\to\,0^+}\,\frac{\log \big(n_0(\varepsilon)+1\big)+\log\frac{n_0(\varepsilon)}{n_0(\varepsilon)+1}+\log\log\frac4{\varepsilon^2}}{\log\frac4{\varepsilon^2}}\\
    &\geqslant \,\limsup_{\varepsilon\,\to\,0^+}\,\frac{\log\frac2{\varepsilon}+\log\frac{n_0(\varepsilon)}{n_0(\varepsilon)+1}+\log\log\frac4{\varepsilon^2}}{\log\frac4{\varepsilon^2}} \,= \,\frac12.
\end{align*}
\end{proof}

\emph{Take $T\colon [-1,1]\to [-1,1]$ such that $\mathrm{\overline{mdim}_{M}}\,\big([-1,1], |\cdot|, T\big)=1$. Regarding the existence of such a map $T$, we refer the reader to  Example~\ref{ex:caixinhas}. Define 
\begin{eqnarray*}
&f\colon \,\,\,K& \quad \to \quad K \\
&\quad \quad \quad \quad (x_n)_{n\,\in\,\mathbb N} & \quad  \mapsto \quad \big(T(x_n)\big)_{n\,\in\,\mathbb N}.
\end{eqnarray*}}
\smallskip

\noindent \emph{\textbf{Claim:}} \emph{$0<\mathrm{\overline{mdim}_{mo}}\,\big(K, d, f\big) \leqslant 1$.}
\medskip

\begin{proof} Fix $0<\delta<1$. Since $\mathrm{\overline{mdim}_{M}}\,\big([-1,1], |\cdot|, T\big)=1$, we can take $\varepsilon_0=\varepsilon_0(\delta)>0$ and $n_0=n_0(\varepsilon_0)$ such that, for any $0<\varepsilon<\varepsilon_0$ and $n>n_0$, one has
\[ S([-1,1],n,\varepsilon)\,\geqslant \, e^{n\,\big(|\log\varepsilon|\, (1-\delta)\big)}.\]
For each $n>n_0$ and each $0<\varepsilon<\varepsilon_0$, let $\{x_1,\dots,x_M\}\subset [-1,1]$ be a maximal $(n,\varepsilon)$-separated set of $[-1,1]$ with respect to the map $T$ (hence, $M= S([-1,1],n,\varepsilon)$). Consider the space 
\[A \,=\,\Big\{(z_n)_{n\,\in\,\mathbb N} \in K \colon \, z_n\in\{x_1,\dots,x_M\} \quad \forall \, 1\leqslant n \leqslant n_0 \,\text{ and } \,z_i = 1 \quad \forall\, i\geqslant \lfloor1/ \varepsilon\rfloor\Big\}.\]
We note that $A$ is a finite subset of $K$ whose cardinality satisfies
$$\# A \, \geqslant \, e^{n\big(|\log\varepsilon| \,(1-\delta)\big)/\varepsilon}.$$ 
Moreover, $A$ is a $(n,\varepsilon^2)$-separated subset of $K$ with respect to the map $f$. Indeed, given $(z_n)_{n\,\in\,\mathbb{N}}\not=(w_n)_{n\,\in\,\mathbb{N}}$ in $A$, there exists $0\leqslant i \leqslant 1/\varepsilon$ such that $z_i\not= w_i$. As $z_i$ and $w_i$ are $(n,\varepsilon)$-separated with respect to the map $T$, there exists $0 \leqslant  N < n$ for which 
$$|T^N(z_i)-T^N(w_i)|\,\geqslant\, \varepsilon.$$
This implies that 
$$\big\|f^{N-1}((z_n)_{n\,\in\,\mathbb{N}})-f^{N-1}((w_n)_{n\,\in\,\mathbb{N}})\big\|\,\geqslant\,  \frac{\varepsilon}{1/\varepsilon}\,=\,\varepsilon^2.$$
Consequently,
\[\limsup_{n\,\to\,+\infty}\,\frac1n \log S(K,n,\varepsilon^2) \,\geqslant \, |\log\varepsilon| \,(1-\delta)/\varepsilon\]
and
\[\mathrm{\overline{mdim}_{mo}}\,\big(K,\|\cdot\|, f\big)\,\geqslant \, \frac12.\]
\end{proof}
\end{example}

Consider now in $[-1,1]$ the Euclidean metric $|\cdot|$ and endow $[-1,1]^\N$ with the product topology. Then, by Theorem~\ref{thm1}, 
$$\mathrm{\overline{mdim}_{mo}}\,\big([-1,1]^\mathbb{N},\rho_{[-1,1]}, f\big) \,\leqslant \, \overline{\mathrm{mo}}\,\big([-1,1]^\N,\rho_{[-1,1]}\big) \,=\, 0 \, < \, \mathrm{\overline{mdim}_{mo}}\,\big([-1,1]^\mathbb{N}, \|\cdot\|, f\big),$$ 
where the equality $\overline{\mathrm{mo}}\,\big([-1,1]^\N,\rho_{[-1,1]}\big) \,=\, 0$ is due to the estimates in \cite[Example~1]{LindTsu}. $\quad \,\,\,\blacksquare$
\smallskip

\begin{example}\textbf{\emph{The inequality in Theorem~\ref{thm1} may be strict.}}\label{ex:id}
\medskip

\noindent \emph{Let $(Y,d)$ be a compact metric space with $0<\mathrm{dim}_B\,\big(Y,d\big) < +\infty$ and consider the space $X=\mathcal{P}(Y)$ endowed with a metric $D \in \{W_p \colon 1 \leqslant p < +\infty\}\cup \{LP\}$. 
As explained in Example~\ref{ex:push}, one has 
$$0 \,<\, \mathrm{mo}\big(\mathcal{P}(Y),D\big) \,=\, \mathrm{dim_B}\,\big(Y,d\big) < +\infty.$$
Let $g \colon \mathcal{P}(Y) \to \mathcal{P}(Y)$ be a continuous map with finite topological entropy (e.g. the identity map). Then, due to  \eqref{eq:leq}, 
$$\mathrm{\overline{mdim}_{mo}}\,
\big(\mathcal{P}(Y),D,g\big) \,=\, 0 \,<\, \mathrm{mo}\big(\mathcal{P}(Y),D\big). \quad \blacksquare$$}
\end{example}

\smallskip


\begin{example}\cite[Remark~18]{BurguetShi} \textbf{\emph{The assumptions in Corollary~\ref{thm2} are sharp.}}
\medskip

\noindent \emph{Let $(Y, d_Y)$ be an infinite compact metric space with zero dimension. Then the space $X = Y^\N$ with the product topology may be endowed with a compatible metric $d$ so that $(X,d)$ has zero box-counting dimension. Consider the full shift $\sigma$ on $X$. Since $Y$ is infinite, $\mathrm{h_{top}}(\sigma) > 0$; and, by Corollary~\ref{cor:push}, 
$$\mathrm{\overline{mdim}_{mo}}\,\big(\mathcal{P}(X),W_1,\sigma_\ast\big) \,\leqslant \,\mathrm{\overline{dim}_B}(X,d) \, = \, 0.$$
So Corollary~\ref{thm2} does not apply. This is due to the fact that $\sigma$ is not Lipschitz on $(X,d)$. $\quad \blacksquare$}
\end{example} 
\smallskip


\begin{example}\cite[Remark~18]{BurguetShi} \textbf{\emph{The inequality in the statement of Theorem~\ref{thm3} is sharp.}} \label{ex:shift}
\medskip

\noindent \emph{Consider the shift map $\sigma \colon [0,1]^{\N} \to [0,1]^{\N}$, where $[0,1]$ is endowed with the Euclidean distance. Then, $\mathrm{\overline{mdim}_M}\,\big([0,1]^{\N}, \rho_{[0,1]}, \sigma\big) = 1$ (cf. \cite[Example~1.1]{LindTsu}). Moreover, $\sigma$ is $2$-Lipschitz: given $(x_k)_{k\,\in\,\N}, (y_k)_{k\,\in\,\N} \in [0,1]^{\N}$, 
\begin{eqnarray*}
\rho_{[0,1]}\big(\sigma((x_k)_{k\,\in\,\N}),\,\sigma((y_k)_{k\,\in\,\N})\big) &=& \rho_{[0,1]}\big((x_{k+1})_{k\,\in\,\N},\,\sigma((y_{k+1})_{k\,\in\,\N})\big) \,=\, \sum_{k=1}^{+\infty} 2^{-k} \,|x_{k+1} - y_{k+1}| \\
&=& 2\,\sum_{k=1}^{+\infty} 2^{-(k+1)} \,|x_{k+1} - y_{k+1}|
\,\leqslant\, 2\,\sum_{k=0}^{+\infty} 2^{-(k+1)} \,|x_{k+1} - y_{k+1}|\\
&=& 2\, \rho_{[0,1]}\big((x_k)_{k\,\in\,\N}, \,(y_k)_{k\,\in\,\N}\big).
\end{eqnarray*}
Since $\sigma$ has infinite topological entropy and is Lipschitz, if $d=\rho_{[0,1]}$ then, by \eqref{eq:Lip}, we must have $$\mathrm{\overline{dim}_B}\,\big([0,1]^{\N},d\big)= +\infty.$$
Note also that, since $d_{n+1} \geqslant d_n$ for every $n \in \N$, one has $S([0,1]^\N,d_n,\vep) \,\leqslant \, S([0,1]^\N, d_{n+1},\vep)$ for every $\vep > 0$; thus,
\begin{equation}\label{eq:dn}
\mathrm{\overline{dim}_B}\,\big([0,1]^\N,d_{n+1}\big)\, \geqslant \, 
\mathrm{\overline{dim}_B}\,\big([0,1]^\N,d_n\big). 
\end{equation}
Therefore, $\mathrm{\overline{dim}_B}\,\big([0,1]^{\N},d_n\big)= +\infty$ for every $n \in \N$ and so 
$$\limsup_{n \, \to \, +\infty}\,
\frac{\mathrm{\overline{dim}_B}\,\big([0,1]^{\N},d_n\big)}{n} \, > \, \mathrm{\overline{mdim}_M}\,\big([0,1]^{\N},\rho_{[0,1]}, \sigma\big). \quad \blacksquare$$}
\end{example}
\smallskip

More generally, if $\mathrm{\overline{dim}_B}\,\big(Y,d_Y\big) > 0$, then $(Y^\N, \rho_Y, \sigma)$ has positive upper metric mean dimension, and so infinite topological entropy. As the shift is a Lipschitz map, we must have  $\mathrm{\overline{dim}_B}\,\big(Y^\N,d_Y\big) = +\infty$. Thus, the mean box-counting dimension of $(Y^\N, d_Y, \sigma)$ is $ +\infty$.

\smallskip

\begin{example}\cite[Example IV-C.2]{GutmanSp} \textbf{\emph{The inequality in the statement of Theorem~\ref{thm3} may be strict, even if the metric mean dimension is zero.}} 
\medskip

\noindent \emph{Given $K \in \N$, consider
\begin{eqnarray*}
A_K &=& \big\{[0, 1/2^K]^K \times \{(0,0,\cdots)\}\big\} \,\subset \, [0,1]^\N\\
\mathcal{A}_K &=& \bigcup_{n \, \in \, \N_0}\, \sigma^n (A_K)\\
\mathcal A &=& \bigcup_{K \, \in \, \N}\, \mathcal{A}_K.
\end{eqnarray*}
Then $\mathcal A$ is a subshift of $([0,1]^\N, \rho_{[0,1]})$ (that is, $\mathcal A$ is a non-empty compact $\sigma$-invariant subset of $[0,1]^\N$) which satisfies the following conditions:}
\smallskip
\smallskip

\noindent $(1)$ \emph{For every $\underline{x} \in \mathcal A$, one has $\lim_{n \, \to \, +\infty}\,\sigma^n(x) = \underline{0} =(0,0,\cdots)$. Thus, the non-wandering set of $\sigma_{|\mathcal A}$ is 
$\{\underline{0}\}$ and so 
$$\mathrm{\overline{mdim}_M}\,\big(\mathcal A,\rho_{[0,1]}, \sigma\big)\,=\, 0.$$
\smallskip}

\noindent \emph{$(2)$ On the other hand, if $d=\rho_{[0,1]}$, then
$$\limsup_{n \, \to \, +\infty}\,
\frac{\mathrm{\overline{dim}_B}\,\big(\mathcal A,d_n\big)}{n} \, \geqslant \, 1.$$
Indeed, in \cite[Appendices A \& E]{GutmanSp} it was established that:
\smallskip
\smallskip}

 \emph{$(2a)$  For every $n \in \N$ and $\vep > 0$, 
$$\# \big(\pi_n(\mathcal A),\,\|\,.\,\|_\infty^n, \,\vep\big) \,\leqslant \,S(\mathcal A, d_n,\vep)$$
where $\pi_n \colon \mathcal A \to [0,1]^n$ is the projection given by
$\pi_n \big((x_j)_{j \, \in \, \N}\big) \, = \, (x_1, \cdots, x_n)$,
$$\|(x_j)_{j \,\in \,\{1, \,\cdots, \,n\}}\|_\infty^n \, = \, \max \,\{x_j\colon 1 \leqslant j \leqslant n\}$$
and $\#(B, \Delta, \varepsilon)$ denotes the minimal cardinality of an open cover $\{U_j\}_j$ of $B$ such that each $U_j$ has diameter in the metric $\Delta$ smaller than $\varepsilon$.}
\smallskip

 \emph{$(2b)$ In addition,
$$\lim_{n \, \to \, +\infty}\,\limsup_{\vep \, \to \, 0^+}\, \frac{\# \big(\pi_n(\mathcal A),\,\|\,.\,\|_\infty^n, \,\vep\big)}{n \, |\log \vep|} \,=\, 1.$$}

\medskip

\noindent \emph{From $(2a)$ and $(2b)$, we conclude that
$$\limsup_{n \, \to \, +\infty}\,
\frac{\mathrm{\overline{dim}_B}\,\big(\mathcal A,d_n\big)}{n} \, = \, \limsup_{n \, \to \, +\infty}\,\limsup_{\vep \, \to \, 0^+}\, \frac{S \big(\mathcal A, d_n, \vep\big)}{n\, |\log \vep|} \,\geqslant \, 1 \,> \,\mathrm{\overline{mdim}_M}\,\big(\mathcal A,\rho_{[0,1]}, \sigma\big). \quad \blacksquare$$}
\end{example} 

\smallskip


\begin{example}\cite[Example IV-C.2]{GutmanSp} \textbf{\emph{There exist two maps with infinite topological entropy, equal metric mean dimension and equal dynamical metric order, but distinct mean box-counting dimension.}} \label{ex:caixinhas}
\medskip

\noindent \emph{Consider the interval $[0,1]$ with the Euclidean distance $d=|\cdot|$ and the following $2$-parameter family of interval maps, first introduced in \cite{KolyadaSnoha} and further discussed in \cite{Hazard, VV, CPV, BCP}. Fix a sequence $a\,=\,(a_k)_{k\,\in \,\N\,\cup\,\{0\}}$ of real numbers and a sequence $b\,=\,(b_k)_{k\,\in \, \N}$ of positive odd integers satisfying the following conditions:
\begin{itemize}
\item $(a_k)_{k\,\geqslant\,0}$ is strictly increasing and converges to $1$.
\smallskip
\item $(a_{k}-a_{k-1})_{k\,\in \, \N}$ decreases to zero.
\smallskip
\item $(b_k)_{k\,\in \, \N}$ is strictly increasing (hence $(b_k)_{k\,\in \, \N}$ has limit $+\infty$).
\end{itemize}}

\smallskip

\noindent \emph{Given $k \in \mathbb N$, denote by $f_k\colon[0,1]\to[0,1]$ the unique continuous piecewise affine map with $b_k$ equidistributed full branches such that $f_k$ is increasing in the first one (and, since $b_k$ is odd, also in the last one). Let 
$$J_k \,=\,[a_{k-1},a_k]$$
and $T_k\colon J_k \to [0,1]$ be the unique increasing affine map from $J_k$ onto $[0,1]$.  Define:
\begin{equation}\label{Tab}
\begin{array}{rccc}
T_{a,b} \colon & [0,1] & \rightarrow & [0,1] \\
& x \in J_k & \mapsto & T_k^{-1}\circ f_k \circ T_k\\
& x=1 & \mapsto & 1
\end{array}
\end{equation}
Each map is determined by a sequence of horseshoes that accumulates at the extreme point $1$ (see Figure~\ref{figure1}). The sequence $(a_k)_k$ defines the intervals $(J_k)_k$ where the horseshoes lie, and each $b_k$ represents the number of branches in the $k$-th horseshoe. In particular, the topological entropy of $T_{a,b}$ restricted to $J_k$ is $\log b_k$. Write $\vep_k = |J_k|/b_k$, which represents the length of the injectivity domains in $J_k$.
\smallskip}

\begin{figure}[!htb] 
\begin{center}
\includegraphics[scale=0.4]{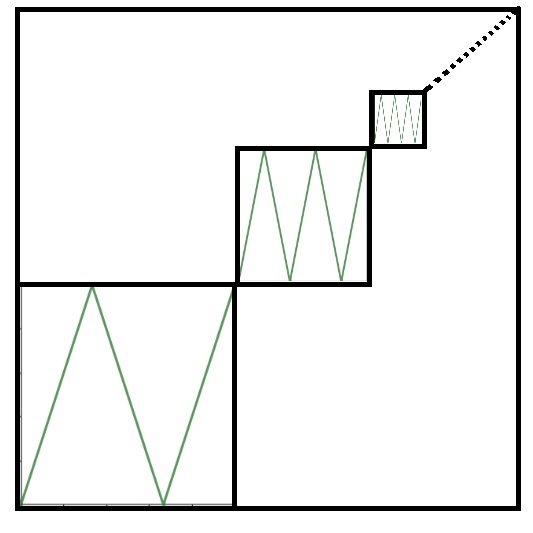}
\caption{Graph of the map $T$ (reproduced from \cite{BCP}).}
\label{figure1}
\end{center}
\end{figure}

\emph{For every pair $a,b$, the map $T_{a,b}$ has infinite topological entropy since a classical result of Misiurewicz (see \cite{Mis}) yields 
$$\mathrm{h_{top}}(T_{a,b})\,=\,\limsup_{k\,\to\,+\infty}\,\log b_k\,=\,+\infty.$$
The metric mean dimension has been computed for specific choices of the parameters (see for instance \cite{VV} and \cite{CPV}). In general, one has (cf. \cite{BCP})
\begin{eqnarray}\label{mismdim}       \nonumber\overline{\mathrm{mdim}}_M([0,1],|\cdot|,T_{a,b})&=&\limsup_{k\,\to\,+\infty}\,\frac{\log b_k}{\log(1/\vep_k)} \\ \nonumber\underline{\mathrm{mdim}}_M([0,1], |\cdot|,T_{a,b})&=&\liminf_{k\,\to\,+\infty}\,\frac{\log b_k}{\log(1/\vep_k)}.
\end{eqnarray}
For instance:}
\smallskip

\noindent \emph{$(1)\,\,$ If $a_k\,=\,\sum_{n=1}^k 6/(\pi n^2)$ and $b_k\,=\,3^k$, then $\mathrm{mdim}_M([0,1],T_{a,b},|\cdot|)\,=\,1$.}
\medskip

\noindent \emph{$(2)\,\,$ Given $0<\beta<1$, if $a_k\,=\, \sum_{n=1}^k C(\beta)\,3^{n\,(1-\frac{1}{\beta})}$, where $C(\beta)\,=\,\big(\sum_{n=1}^{+\infty} 3^{n\,(1-\frac{1}{\beta})}\big)^{-1}$, and $b_k\,=\,3^k$, then $\mathrm{mdim}_M([0,1],T_{a,b},|\cdot|)\,=\,\beta$.}
\medskip

\noindent \emph{$(3)\,\,$   If $a_k\,=\, 1/2^{k}$ and $b_k\,=\,2k+1$, then $\mathrm{mdim}_M([0,1],T_{a,b},|\cdot|)\,=\,0$.}

\medskip

\noindent \emph{We observe that, since $\mathrm{dim}_B([0,1]) = 1 < +\infty$, by Lemma~\ref{le:mo1} and Theorem~\ref{thm1} one has
$$\mathrm{mdim_{mo}}\,\big([0,1], T_{a,b}, |\cdot|\big) = 0.$$}

\emph{Regarding example $(1)$ and its mean box-counting dimension, Theorem~\ref{thm3} yields
\begin{equation}\label{eq:mbd1}
\limsup_{n \, \to \, +\infty}\,
\frac{\mathrm{\overline{dim}_B}\,\big([0,1],d_n\big)}{n} \, \geqslant \, \mathrm{mdim}_M([0,1],T_{a,b},|\cdot|) \,=\, 1.
\end{equation}
On the other hand (cf. \cite[Proposition~10.1]{CRV3}), for all $n \in \N$ and $0 < \vep < \vep_k$,
$$S(T_{a,b}, n, \vep) \, \leqslant \, \Big\lfloor \frac{1}{\vep_k}\Big\rfloor^n\, \times \, \frac{\vep_k}{\vep}$$
so, since 
$$\Big\lfloor \frac{1}{\vep_k} \Big\rfloor\, \times \, \vep_k \, \leqslant \, \frac{1}{\vep_k}\, \times \, \vep_k \, = 1$$
we have
$$
\log S(T_{a,b}, n, \vep) \,\leqslant \, \log \Big(\Big\lfloor \frac{1}{\vep_k}\Big\rfloor^{n-1}\Big)\, + \, \log \frac{1}{\vep} \,=\, (n-1) \, \log \Big\lfloor \frac{1}{\vep_k}\Big\rfloor \, + \, \log \frac{1}{\vep} \,\leqslant \, n \, \log \frac{1}{\vep}.$$
Thus,
$$\limsup_{\vep \, \to \, 0^+}\,\frac{\log S(T_{a,b}, n, \vep) }{|\log \vep|}\, \leqslant n$$
hence, 
\begin{equation}\label{eq:mbd2}
\limsup_{n \, \to \, +\infty}\,\frac{\limsup_{\vep \, \to \, 0^+}\,\frac{\log S(T_{a,b}, n, \vep)}{|\log \vep|}}{n}\, \leqslant 1.
\end{equation}
From \eqref{eq:mbd1} and \eqref{eq:mbd2}, we conclude that
$$\limsup_{n \, \to \, +\infty}\,
\frac{\mathrm{\overline{dim}_B}\,\big([0,1],d_n\big)}{n} \,=\, 1.$$}

\medskip

\emph{Consider now the shift map $\sigma \colon [0,1]^{\N} \to [0,1]^{\N}$ discussed on Example~\ref{ex:shift}. We know that $\mathrm{\overline{mdim}_M}\,\big([0,1]^{\N}, \rho_{[0,1]},  \sigma\big) = 1$, hence $h_{\mathrm{top}}(\sigma)=+\infty$,  and, by Theorem~\ref{thm1} and Lemma~\ref{le:mo1}, 
$$\mathrm{mdim_{mo}}\,\big([0,1]^\N,  \rho_{[0,1]},  \sigma\big) \,=\, \mathrm{mo}\big([0,1], |\cdot|\big) \, = \, 0.$$
Yet, we have already verified (cf. Example~\ref{ex:shift}) that the mean box-counting dimension of $([0,1]^\N, \sigma, \rho_{[0,1]})$ is $+\infty$. 
$ \quad \blacksquare$}
\end{example}

\section{Basic properties of the dynamical metric order}\label{se:basic}

In what follows, it will be useful to recall that, given two sequences $(a_n)_{n \, \in\, \N}$ and $(b_n)_{n \, \in\, \N}$ of real numbers, then
\begin{equation}\label{eq:sum}
\limsup_{n\, \to\, +\infty}\, (a_n + b_n) \,\leqslant\, \limsup_{n\, \to\, +\infty}\, a_n + \limsup_{n\, \to\, +\infty}\, b_n    
\end{equation}
and that, if $a_n \geqslant 0$ and $b_n \geqslant 0$ for every $n \in \N$ (or for $n$ large enough), 
\begin{equation}\label{eq:prod}
\limsup_{n\, \to\, +\infty}\, (a_n \cdot b_n) \,\leqslant\, \big(\limsup_{n\, \to\, +\infty}\, a_n\big) \cdot \big(\limsup_{n\, \to\, +\infty}\, b_n\big)    
\end{equation}
unless the product at the right-hand side of this inequality is $\infty \cdot 0$ or $0 \cdot \infty$.

\subsection{Power of a map}
Given a continuous map $f\colon X\to X$ on the compact metric space $(X,d)$, it is known (cf. \cite{Wa}) that, for every $\ell \in \N$,
$$\mathrm{h_{top}}(f^\ell) = \ell \cdot \mathrm{h_{top}}(f)$$
and (cf. \cite{LW2000})
$$\mathrm{\overline{mdim}_{M}}\,\big(X,d,f^\ell\big)\, \leqslant\,\ell \cdot \mathrm{\overline{mdim}_{M}}\,\big(X,d,f\big).$$
Due to the impact of applying the map $\log$ twice to define the  metric order, a formula for the dynamical metric order of the power of a map is slightly different.

\begin{proposition}
Let $f\colon X\to X$ be a continuous map acting on the compact metric space $(X,d)$ such that 
$0 < \mathrm{\overline{mdim}_{mo}}\,\big(X,d,f\big) < +\infty$. Then, for any $\ell\in\mathbb N$ there exists $C_\ell\in [0,1]$ such that
    \[C_\ell\cdot\mathrm{\overline{mdim}_{mo}}\,\big(X,d,f\big)\, \leqslant\,\mathrm{\overline{mdim}_{mo}}\,\big(X,d,f^\ell\big)\,
    \leqslant \, \mathrm{\overline{mdim}_{mo}}\,\big(X,d,f\big).\]
If, in addition, $f$ is a Lipschitz map, then $C_\ell=1$ and 
    \[\mathrm{\overline{mdim}_{mo}}\,\big(X, d,f\big) \,=\,\mathrm{\overline{mdim}_{mo}}\,\big(X,d,f^\ell\big).\]
\end{proposition}
\smallskip

\begin{proof} 
Fix a continuous map $f\colon X\to X$ and $\ell\in\mathbb N$.  To distinguish the separated sets of $f\colon X \to X$ from the ones of $f^\ell\colon X \to X$, for the moment we denote by $S(Y,n,\varepsilon, g)$ the maximal cardinality of the $(n,\vep)$-separated subsets of $Y$ under the action of a map $g\colon Y \to Y$.

Given $\varepsilon>0$, 
$$S(X,n,\varepsilon,f^\ell)\,\leqslant\, S(X,\ell n,\varepsilon, f) \quad \quad \forall\, n\in\mathbb N$$
and so
\begin{eqnarray*}
\limsup_{n\, \to\, +\infty}\,\frac{1}{n}\,\log S(X,n,\varepsilon, f^\ell)  &\leqslant& \ell \cdot\limsup_{n\, \to\, +\infty}\,\frac{1}{\ell n}\,\log S(X,\ell n,\varepsilon,f) \nonumber \\
\frac{\log^+ \big(\limsup_{n\, \to\, +\infty}\,\frac{\log S(X,\,n,\,\varepsilon, \,f^\ell)}{n}\big)}{|\log \vep|}  & \leqslant &  \frac{\log \ell}{|\log \vep|} +  \,\frac{\log^+ \big(\limsup_{n\, \to\, +\infty}\,\frac{\log S(X,\,\ell n,\,\varepsilon,\,f)}{\ell n}\big)}{|\log \vep|}.
\end{eqnarray*}
Therefore, using \eqref{eq:sum} we get
\begin{equation} \label{eq:power}
\mathrm{\overline{mdim}_{mo}}\,\big(X,d,f^\ell\big) \,\leqslant \,\mathrm{\overline{mdim}_{mo}}\,\big(X,d,f\big).
\end{equation}

\smallskip
Notice that, to deduce \eqref{eq:power} we do not need any assumption on $\mathrm{\overline{mdim}_{mo}}\,\big(X,d,f\big)$. Yet, if we suppose that 
$\mathrm{\overline{mdim}_{mo}}\,\big(X,d,f\big) < +\infty$, then \eqref{eq:power} yields 
$\mathrm{\overline{mdim}_{mo}}\,\big(X,d,f^\ell\big) < +\infty$ and, in particular, if 
$\mathrm{\overline{mdim}_{mo}}\,\big(X,d,f\big) = 0$, one also has  
$\mathrm{\overline{mdim}_{mo}}\,\big(X,d,f^\ell\big) = 0$.
\smallskip

On the other hand, since $f$ is continuous and $X$ is compact, given $\varepsilon>0$ there exists $\delta(\varepsilon)>0$ such that
\begin{align*}
d(x,y)\,<\,\delta(\vep) \quad \Rightarrow \quad \max\big\{d(f^i(x),f^i(y)) \colon \,\, 0 \leqslant i \leqslant  \ell\big\} < \varepsilon.
\end{align*}
Thus, for every $n \in \N$,
\[S(X,n,\delta(\varepsilon),f^\ell)\,\geqslant \,  S(X, n,\varepsilon,f).\]
So, 
\begin{eqnarray*}
\limsup_{n\, \to\, +\infty}\,\frac{1}{n}\,\log S(X,n,\delta(\vep), f^\ell)  &\geqslant& \limsup_{n\, \to\, +\infty}\,\frac{1}{n}\,\log S(X,n,\varepsilon,f) \\
\frac{\log^+ \big(\limsup_{n\, \to\, +\infty}\,\frac{\log S(X,\,n,\,\delta(\vep), \,f^\ell)}{n}\big)}{|\log \vep|}  & \geqslant &  \frac{\log^+ \big(\limsup_{n\, \to\, +\infty}\,\frac{\log S(X,\, n,\,\varepsilon,\,f)}{n}\big)}{|\log \vep|}\\
\frac{\log^+ \big(\limsup_{n\, \to\, +\infty}\,\frac{\log S(X,\,n,\,\delta(\vep), \,f^\ell)}{n}\big)}{|\log \delta(\vep)|}\cdot \frac{|\log \delta(\vep)|}{|\log \vep|} 
& \geqslant &  \frac{\log^+ \big(\limsup_{n\, \to\, +\infty}\,\frac{\log S(X,\, n,\,\varepsilon,\,f)}{n}\big)}{|\log \vep|}.
\end{eqnarray*}
Letting
$$\Delta_\ell \, = \, \displaystyle\limsup_{\varepsilon\,\to\,0^+}\,\frac{\log \delta(\varepsilon)}{\log \varepsilon}$$
and taking into account that $\mathrm{\overline{mdim}_{mo}}\,\big(X,d,f^\ell\big) < +\infty$, we deduce by applying \eqref{eq:prod} that 
\begin{equation*}
\Delta_\ell\cdot\mathrm{\overline{mdim}_{mo}}\,\big(X,d,f^\ell\big)\,\geqslant \, \mathrm{\overline{mdim}_{mo}}\,\big(X,d,f\big).
\end{equation*}
Moreover, since $\mathrm{\overline{mdim}_{mo}}\,\big(X,d,f\big) > 0$ by assumption, then we must have
$\Delta_\ell >0$. Consequently, if $C_\ell = 1/\Delta_\ell$, then
\begin{equation*}
\mathrm{\overline{mdim}_{mo}}\,\big(X,d,f^\ell\big)\,\geqslant \, C_\ell\cdot\mathrm{\overline{mdim}_{mo}}\,\big(X,d,f\big).
\end{equation*}
\smallskip

When $f$ is Lipschitz, that is, if there is a constant $L > 0$ such that
$$d(f(x),\,f(y))\,\leqslant \, L \, d(x,y) \quad \forall\, x,y \in X$$
then we can choose 
$$\delta(\vep)\,=\, \frac{1}{2}\, \min\,\Big\{\frac{\vep}{L^j} \colon \, 0 \leqslant j \leqslant \ell\Big\}$$
and so $C_\ell=1$ and $\mathrm{\overline{mdim}_{mo}}\,\big(X,d,f^\ell\big) = \mathrm{\overline{mdim}_{mo}}\,\big(X,d,f\big)$, as claimed. 
\end{proof}

\subsection{Product of maps}
Given continuous maps $f\colon X\to X$ and $g\colon Y\to Y$ acting on compact metric spaces $(X,d_X)$ and $(Y,d_Y)$, respectively, let $Z=X\times Y$ be the product space endowed with the product metric 
$$d_Z\big((x_1,y_1),(x_2,y_2)\big) \, = \, \max\,
\big\{d_X(x_1,x_2),d_Y(y_1,y_2)\big\}$$
and $f\times g\colon Z\to Z$ be the map defined by $(f\times g)(x,y) = \big(f(x),g(y)\big)$. It is known (cf. \cite{Wa}) that
$$h_{top}(f\times g) = h_{top}(f) + h_{top}(g)$$
and (cf. \cite{LW2000})
$$\mathrm{\overline{mdim}_{M}}\,\big(Z,d_Z,f\times g\big)\, \leqslant\, \mathrm{\overline{mdim}_{M}}\,\big(X,d_X,f\big) + \mathrm{\overline{mdim}_{M}}\,\big(Y,d_Y,g\big).$$
Under a mild assumption, a similar relation is valid for the dynamical metric order.

\begin{proposition}
Let $f\colon X\to X$ and $g\colon Y\to Y$ be continuous maps with infinite topological entropy, acting on compact metric spaces $(X,d_X)$ and $(Y,d_Y)$, respectively. Then,
\begin{equation}\label{eq:product}
\mathrm{\overline{mdim}_{mo}}\,\big(Z, f\times g,d_Z\big)\,\leqslant \, \mathrm{\overline{mdim}_{mo}}\,\big(X,d_X,f\big) + \mathrm{\overline{mdim}_{mo}}\,\big(Y,d_Y,g\big).
\end{equation}
\end{proposition}

\begin{proof}
Firstly, note that, for any $\varepsilon>0$ and $n \in \N$, 
$$R(Z,n,\varepsilon)\,\leqslant \, R(X,n,\varepsilon)\cdot R(Y,n,\varepsilon)$$
and so, by \eqref{eq:sum},
$$\limsup_{n\, \to\, +\infty}\,\frac{1}{n}\,\log R(Z,n,\varepsilon) \, \leqslant \,\limsup_{n\, \to\, +\infty}\,\frac{1}{n}\,\log R(X,n,\varepsilon) \,+ \,\limsup_{n\, \to\, +\infty}\,\frac{1}{n}\,\log R(Y,n,\varepsilon).$$
If these three values are finite, we can apply the function $\log^+$ to both sides of this inequality, and get
\begin{eqnarray}\label{eq:times}
&& \frac{\log^+ \big(\limsup_{n\, \to\, +\infty}\,\frac{1}{n}\,\log R(Z,n,\varepsilon)\big)}{|\log \vep|}  \nonumber \\
&\leqslant& \frac{\log^+ \big(\limsup_{n\, \to\, +\infty}\,\frac{1}{n}\,\log R(X,n,\varepsilon) \,+ \,\limsup_{n\, \to\, +\infty}\,\frac{1}{n}\,\log R(Y,n,\varepsilon)\big)}{|\log \vep|}.
\end{eqnarray}
\smallskip

\noindent From \eqref{eq:times} we easily conclude that:\\

\noindent $\bullet$ If $\mathrm{h_{top}}(f)$  and $\mathrm{h_{top}}(g)$ are both finite, then there exist $A > 0$ and $\varepsilon_0>0$ such that $h_\varepsilon(f),h_\varepsilon(g)\leqslant A$ for every $0<\varepsilon\leqslant\varepsilon_0$. By \eqref{eq:times}, we obtain
\begin{align*}
\mathrm{\overline{mdim}_{mo}}\,\big(Z, f\times g,d_Z\big)
&  \,=\,\mathrm{\overline{mdim}_{mo}}\,\big(X,d_X,f\big)+\,\mathrm{\overline{mdim}_{mo}}\,\big(Y,d_Y,g\big)=0.
\end{align*}

\noindent $\bullet$ If $0<\mathrm{h_{top}}(f)<\mathrm{h_{top}}(g)$, then 
there exists $\vep_0>0$ such that, for every $0 < \vep < \vep_0$, one has $0<h_\varepsilon(f)<h_\varepsilon(g)$. By \eqref{eq:times}, 
\begin{eqnarray*}
\frac{\log^+ \big(\limsup_{n\, \to\, +\infty}\,\frac{1}{n}\,\log R(Z,n,\varepsilon)\big)}{|\log \vep|} &\leqslant&\frac{\log^+(2h_{\varepsilon}(g))}{|\log \vep|}
\end{eqnarray*}
so
\begin{align*}
\mathrm{\overline{mdim}_{mo}}\,\big(Z, d_Z,f\times g\big)
& \leqslant \, \,\mathrm{\overline{mdim}_{mo}}\,\big(Y,d_Y,g\big)\leqslant\,\mathrm{\overline{mdim}_{mo}}\,\big(X,d_X,f\big)+\,\mathrm{\overline{mdim}_{mo}}\,\big(Y,d_Y,g\big).
\end{align*}

\smallskip

\noindent However, since the logarithm is a concave function, 
$\log \,(a+b)$ is, in general, not smaller than $\log a + \log b$.
Thus, we cannot conclude \eqref{eq:product} from \eqref{eq:times}. 
Nevertheless, we can derive \eqref{eq:product} whenever $f$ and $g$ have infinite topological entropy. 
\smallskip

Assume that
$$\lim_{\vep \, \to \, 0^+}\,\limsup_{n\, \to\, +\infty}\,\frac{1}{n}\,\log S(X,n,\varepsilon) \, = \,+\infty \, = \,  \lim_{\vep \, \to \, 0^+}\,\limsup_{n\, \to\, +\infty}\,\frac{1}{n}\,\log S(Y,n,\varepsilon).$$
Then we can find  $\varepsilon_0>0$ such that, for every $0<\varepsilon<\varepsilon_0$,  
\begin{equation}\label{eq:22}
\limsup_{n\, \to\, +\infty}\,\frac{1}{n}\,\log S(X,n,\varepsilon) \, > \, 2 \quad \quad \text{ and } \quad \quad \limsup_{n\, \to\, +\infty}\,\frac{1}{n}\,\log S(Y,n,\varepsilon) \,>\,2.    
\end{equation}
\smallskip

\noindent \textbf{Claim}: \emph{If $a \geqslant 2$ and $b \geqslant 2$, then $\,\log (a+b) \,\leqslant \,\log (a\,b) = \log a + \log b.$}
\smallskip

\noindent Indeed, 
$\,a \geqslant 2,\,\, b \geqslant 2\,  \Rightarrow \, (a - 1)(b - 1) \geqslant 1 \,  \Leftrightarrow \, ab - a - b + 1 \geqslant 1 \, \Leftrightarrow \, ab \geqslant  a + b.$\\

\medskip

Summoning \eqref{eq:times}, \eqref{eq:22} and the previous claim, we get
\begin{eqnarray*}
&& \frac{\log^+ \big(\limsup_{n\, \to\, +\infty}\,\frac{1}{n}\,\log R(Z,n,\varepsilon)\big)}{|\log \vep|}  \nonumber \\
&\leqslant& \frac{\log^+ \big(\limsup_{n\, \to\, +\infty}\,\frac{1}{n}\,\log R(X,n,\varepsilon) \,+ \,\limsup_{n\, \to\, +\infty}\,\frac{1}{n}\,\log R(Y,n,\varepsilon)\big)}{|\log \vep|}\\
&\leqslant& \frac{\log^+ \big(\limsup_{n\, \to\, +\infty}\,\frac{1}{n}\,\log R(X,n,\varepsilon)\big)}{|\log \vep|} \,+ \,\frac{\log^+ \big(\limsup_{n\, \to\, +\infty}\,\frac{1}{n}\,\log R(Y,n,\varepsilon)\big)}{|\log \vep|}.
\end{eqnarray*}
Therefore, applying \eqref{eq:sum} we conclude that
$$
\mathrm{\overline{mdim}_{mo}}\,\big(Z, d_Z,f\times g\big) 
\, \leqslant \, 
\mathrm{\overline{mdim}_{mo}}\,\big(X,d_X,f\big)+\mathrm{\overline{mdim}_{mo}}\,\big(Y,d_Y,g\big).$$
\end{proof}

\subsection{Restriction to the non-wandering set}

Given a continuous map $f\colon X\to X$ on a compact metric space $(X,d)$, a point $x \in X$ is said to be non-wandering for $f$ if for every open neighborhood $U$ of $x$ there exists $n \in \N$ such that $f^n(U) \cap U \neq \emptyset$. It is well known that the topological entropy (respectively, the metric mean dimension) of $f$ is equal to the topological entropy (respectively, the metric mean dimension) of the restriction of $f$ to the non-wandering set of $f$ (denoted by $\Omega(f)$).

\begin{proposition}
Let $f\colon X\to X$ be a continuous map acting on a compact metric space $(X,d)$. Then
\[\mathrm{\overline{mdim}_{mo}}\,\big(X,d,f\big) \,=\, \mathrm{\overline{mdim}_{mo}}\,\big(\Omega(f), d,f|_{\Omega(f)}\big).\]
\end{proposition}

\begin{proof}
By definition, one has 
$$\mathrm{\overline{mdim}_{mo}}\,\big(X,d,f\big) \,\geqslant\, \mathrm{\overline{mdim}_{mo}}\,\big(\Omega(f),d,f|_{\Omega(f)}\big).$$
The converse inequality
\begin{align*}\label{ineq:vital}
    \mathrm{\overline{mdim}_{mo}}\,\big(X,d,f\big) \,\leqslant\, \mathrm{\overline{mdim}_{mo}}\,\big(\Omega(f),d,f|_{\Omega(f)}\big).
\end{align*}
is an immediate consequence of the fact that, for every $\varepsilon>0$, one has (cf. \cite{Bowen1})
\[h_{2\varepsilon}(f)\,\leqslant \,h_{\varepsilon}(f|_{\Omega(f)})\]
where $h_\vep$ is the topological entropy at scale $\vep$ defined in \eqref{eq:hvep}.
\end{proof}

\subsection{Bi-Lipschitz invariance}
The dynamical metric order is invariant under a bi-Lipschitz conjugation. The proof of this property is entirely analogous to the one done for the metric mean dimension since a bi-Lipschitz homeomorphism changes distances and scales by a uniform constant.

\section{Proof of Theorem~\ref{thm1}}\label{se:proofA}

Let $(X,d)$ be a compact metric space whose diameter (which we denote by $\mathrm{diam}(X)$) is positive.  
\medskip

\noindent $(a)$ Given $\varepsilon>0$, take a maximal $2\varepsilon$-separated set $\{x_1,\dots,x_{m}\}$ in $X$; hence $m = S(X,d,2\varepsilon)$. Let $n \in \N$ and fix $x_0 \in X$. Any two distinct points in the set
\[A=\Big\{(y_k)_{k\,\in\,\mathbb{N}}\in X^\mathbb{N} \colon \,y_k \in \{x_1,\dots,x_m\} \,\,\text{ for all $1 \leqslant k < n$, and $y_k = x_0$ otherwise}\Big\}\]
are $\varepsilon$-separated in the metric $(\rho_X)_n$. Consequently,
\[S(X^\N,n,\varepsilon)\,\geqslant \, S(X,d,2\varepsilon)^n\]
and so
\[\limsup_{n\,\to\,+\infty}\,\frac{1}{n}\,\log S(X^\N,n,\varepsilon)\,\geqslant\, \log S(X,d,2\varepsilon).\]
Thus,
$$\mathrm{\overline{mdim}_{mo}}\,\big(X^\mathbb{N}, \rho_{_X}, \sigma\big)\,\geqslant\, \mathrm{\overline{mo}}(X,d).$$

\medskip

Regarding the opposite inequality, given $\varepsilon > 0$, take $N=N(\varepsilon) \in \N$ such that 
$$\sum_{k\,>\,N} \frac{\mathrm{diam}(X)}{2^{k}}<\varepsilon/2.$$
Fix $x_0 \in X$ and let $\{x_1,\dots,x_\ell\}$ be an $\varepsilon$-spanning subset of $X$, with minimal cardinality $\ell = R(X, d,\vep)$, and
\[B \,=\,\Big\{(y_k)_{k\,\in\,\mathbb{N}}\in X^\mathbb{N} \colon y_k\in\{x_1,\dots,x_\ell\} \,\text{for all  $1\leqslant k < n+N(\varepsilon)$ and $y_k=x_0$ otherwise} \Big\}.\]
Then $B$ is an $(n,\varepsilon)$-spanning subset of  $X^\mathbb{N}$ with cardinality  
$R(X, d,\vep)^{n+N(\varepsilon)}$. Thus,
$$\limsup_{n\,\to\,+\infty}\,\frac{1}{n}\,\log R(X^\N,n,\varepsilon) \,\leqslant \, \log R(X, d,\vep)$$
and so
$$\mathrm{\overline{mdim}_{mo}}\,\big(X^\mathbb{N}, \rho_{_X}, \sigma\big)\,\leqslant \, \mathrm{\overline{mo}}(X,d).$$
\smallskip

\noindent $(b)$ Given a continuous transformation $f\colon X \to X$,  consider the map 
\begin{align*}
    h\colon X &\,\, \to \,\,X^{\mathbb N}\\
    x &\,\, \mapsto \,\,\big(f^{j-1}(x)\big)_{j\,\in\, \mathbb{N}}.
\end{align*}
It is straightforward to verify that $h$ is continuous, one-to-one over its image and satisfies $h\circ f = \sigma\circ h$. As $X$ is compact, $h$ is a topological conjugacy between $(X,f)$ and $(h(X),\sigma)$.  We claim that $$\mathrm{\overline{mdim}_{mo}}\,\big(h(X),\rho_{_X}, \sigma
\big)\,\geqslant \,\mathrm{\overline{mdim}_{mo}}\,\big(X,d,f\big).$$
Indeed, given an $(n,\varepsilon)$-separated set $E \subset X$ in the metric $d$, then $h(E)$ is an $(n,\varepsilon/2)$-separated set in $(X^\mathbb{N}, \rho_{_X})$. This yields the claimed inequality and completes the proof of Theorem~\ref{thm1}.

\section{Proof of Corollary~\ref{thm4}}\label{se:proofD}

Given $n\in\mathbb{N}$ and $\vep>0$, by Lemma~\ref{lemma0} and Lemma~\ref{lemma01} one has 
\[Q_{\mu,D_n}(\vep) \,\leqslant \, \mathcal{N}(X,n,\vep) \quad \forall\, \mu \in \cP_f(X).\]
Therefore,
\begin{equation}\label{eq:leqsup}
\mathrm{\overline{mdim}_{mo}}\,\big(X,d,f\big) \, \geqslant \, \sup_{\mu\, \in\, \cP_f(X)}\,
\mathrm{\overline{mdim}_{mo}}\,\big(X,d,f,D,\mu\big).
\end{equation}

For the converse inequality, we can adapt the proofs of \cite[Theorem~A and B]{CP2026} and \cite[Theorem~1.3]{CYZ} by just applying an extra $\log$ before dividing by $|\log \varepsilon|$ and taking the final limit with the scale $\varepsilon \to 0^+$. This reasoning yields, when $f$ is a homeomorphism, the existence of an $f$-invariant Borel probability measure $\mu_0$ such that 
\begin{equation}\label{eq:geq}
\mathrm{\overline{mdim}_{mo}}\,\big(X,d,f\big) \, \leqslant \, 
\mathrm{\overline{mdim}_{mo}}\,\big(X,d,f,D,\mu_0\big).
\end{equation}
The proof of Corollary~\ref{thm4} is complete by bringing together \eqref{eq:leqsup} and \eqref{eq:geq}.

\section{Proof of Corollary~\ref{thm2}}\label{se:proofC}

Let $(X,d)$ be a compact metric space with $\mathrm{\overline{dim}}_B(X,d) \, < \, +\infty$ and $f\colon X \to X$ be a continuous map. The upper bound for the dynamical metric order of the induced map $f_*$ stated in Corollary~\ref{thm2} was already obtained in Corollary~\ref{cor:push}, which yields 
\begin{equation}\label{eq:finite}
\mathrm{\overline{mdim}_{mo}}\,\big(\mathcal{P}(X),W_1,f_\ast\big) \,\leqslant\, \mathrm{\overline{dim}}_B(X,d) \, < \, +\infty.
\end{equation}

The lower bound for the dynamical metric order of the induced map announced in Corollary~\ref{thm2} is a direct consequence of \cite[Lemma 17]{BurguetShi}, which states that, if $(X,d)$ is a compact metric space with finite upper box-counting dimension and $f\colon X \to X$ is a continuous $\lambda$-Lipschitz map with positive entropy (hence, $h_{top}(f_\ast)) = +\infty$ and we may replace $\log^+$ by $\log$), then there exists $\alpha > 0$  such that
$$\forall \,0 < \vep < 1 \quad \quad  \frac{\log\big(\limsup_{n \, \to \, +\infty}\,\frac{1}{n}\,\log S(\mathcal P(X),n,\vep)\big)}{|\log \vep|} \, \geqslant \, \alpha$$
where $\mathcal P(X)$ is endowed with the metric $W_1$. More precisely, let
$$\mathcal{L} = 2 \log \lambda + \frac{\mathcal{K}\, \log 2}{2}$$
where 
$\mathcal{K}$ is a positive constant, depending on $X$ and $f$, provided by \cite[Proposition~3.9 (2))]{KerrLi}. 

\begin{lemma}\cite[Lemma 17]{BurguetShi}\label{le:lowerbound}
Consider the space $\mathcal P(X)$ endowed with the metric $W_1$. For any $0 < \gamma  < \frac{\mathcal{K}\, \log 2}{2\, \mathcal{L}}$, there exists a constant $C > 0$ such that
$$\forall \, 0 < \vep < 1 \quad \quad \frac{\log\big(\limsup_{n \, \to \, +\infty}\,\frac{1}{n}\,\log S(\mathcal P(X),n,\vep)\big)}{|\log \vep|} \,\geqslant\, C \, \vep^{-\gamma}.$$
\end{lemma}

In particular, from Lemma~\ref{le:lowerbound} we conclude that, if $(X,d)$ is a compact metric space with finite upper box-counting dimension and $f\colon X \to X$ is a continuous $\lambda$-Lipschitz map with positive entropy, then 
$$\limsup_{\vep \, \to \, 0^+}\, \frac{\log\big(\limsup_{n \, \to \, +\infty}\,\frac{1}{n}\,\log S(\mathcal P(X),n,\vep)\big)}{|\log \vep|} \,\geqslant\, \frac{\mathcal{K}\, \log 2}{2\, \mathcal{L}}$$
that is,
\begin{equation}\label{eq:positive}
\mathrm{\overline{mdim}_{mo}}\,\big(\mathcal{P}(X),W_1,f_\ast\big) \,\geqslant\, \frac{\mathcal{K}\, \log 2}{2\, \mathcal{L}} \, > \, 0.
\end{equation}
Inequalities \eqref{eq:finite} and \eqref{eq:positive} together complete the proof of Corollary~\ref{thm2}.

\section{Proof of Theorem~\ref{thm3}}\label{se:proofCC}

We start by showing that 
$$\limsup_{n \, \to \, +\infty}\,
\frac{\mathrm{\overline{mo}}\,\big(\mathcal{P}(X), W_{1,n}\big)}{n} \,=\, \limsup_{n \, \to \, +\infty}\,
\frac{\mathrm{\overline{dim}_B}\,\big(X,d_n\big)}{n}.$$
Recall that, given $n \in \N$ and $\vep >0$, we denote by $\mathcal{N}(\mathcal{P}(X), W_{1,n}, \vep)$ the smallest number of open balls of radius $\vep$ in the distance $W_{1,n}$ needed to cover $\mathcal{P}(X)$. By \cite[Theorem~A.1]{BGV}, for every $\vep>0$ and $n \in \N$ one has
$$\mathcal{N}(\mathcal{P}(X), W_{1,n}, \vep) \, \leqslant \, \left(\frac{C}{\vep}\right)^{\mathcal{N}(X,\, d_n,\,\vep/2)}$$
where $C>0$ is a multiple of the diameter of $(X,d).$ Therefore,
$$\frac{\log \log \mathcal{N}(\mathcal{P}(X),  W_{1,n},\vep)}{n\,|\log \vep|} \, \leqslant \, \frac{\log \mathcal{N}(X,  d_n,\vep/2)}{n\,|\log \vep|} \,+\, \frac{\log \log \big(\frac{C}{\vep}\big)}{n\,|\log \vep|}.$$
So,
$$\limsup_{\vep \, \to \, 0^+}\,\frac{\log \log \mathcal{N}(\mathcal{P}(X), W_{1,n}, \vep)}{n\,|\log \vep|} \, \leqslant \, \frac{\mathrm{\overline{dim}}_B(X,d_n)}{n}.$$
Consequently,
$$\limsup_{n \, \to \, +\infty}\,\limsup_{\vep \, \to \, 0^+}\,\frac{\log \log \mathcal{N}(\mathcal{P}(X),  W_{1,n},\vep)}{n\,|\log \vep|} \, \leqslant \, \limsup_{n \, \to \,+\infty}\,\frac{\mathrm{\overline{dim}}_B(X,d_n)}{n}$$
that is, 
\begin{equation}\label{eq:1}
\limsup_{n \, \to \, +\infty}\,
\frac{\mathrm{\overline{mo}}\,\big(\mathcal{P}(X), W_{1,n}\big)}{n} \,\leqslant\, \limsup_{n \, \to \, +\infty}\,
\frac{\mathrm{\overline{dim}_B}\,\big(X,d_n\big)}{n}.
\end{equation}
\smallskip

We now address the converse inequality. Following \cite{BB}, given $n \in \N$ and $\vep > 0$, one says that two Borel probability measures on $(X,d)$ are $(n,\vep)$-apart if their supports distance at least $\vep$ in the Hausdorff distance determined by the metric $d_n$. More precisely, $\mu, \nu \in \mathcal{P}(X)$ are $(n,\vep)$-apart if and only if 
$$\min\big\{d_n(x,y)\colon \,x \in \mathrm{supp}(\mu),\; y \in \mathrm{supp}(\nu)\big\} \,\geqslant \, \vep.$$
Denote by $A(\mathcal P(X), d_n, \vep)$ the maximal number of pairwise $(n,\vep)$-apart Borel probability measures in $\mathcal P(X)$. Observe that, if $\{x_1, \ldots, x_k\} \subset X$ is $(n, \vep)$-separated, then the set of Dirac measures on those points, that is, $\{\delta_{x_1}, \cdots, \delta_{x_k}\} \subset \mathcal{P}(X)$, is made of pairwise $(n, \vep)$-apart Borel probability measures. Thus,
\begin{equation}\label{eq:sA}
S(X, d_n,\vep) \,\leqslant \, A(\mathcal P(X), d_n,\vep).
\end{equation}

On the other hand, from the proof of \cite[Theorem 1.6]{BB} we conclude that, for every $\vep>0$ and $n \in \N$, one has
$$\frac{\log \log S(\mathcal{P}(X), W_{1,n},\vep/4)}{n} \, \geqslant \, \frac{\log c + \log \big(A(\mathcal{P}(X), d_n, \vep)\big)}{n}$$
where $c>0$ is a constant. So,
$$\limsup_{\vep\, \to \, 0^+}\,\frac{\log \log S(\mathcal{P}(X), W_{1,n},\vep/4)}{n \,|\log \vep|} \, \geqslant \, \limsup_{\vep\, \to \, 0^+}\,\frac{\log \big(A(\mathcal{P}(X), d_n, \vep)\big)}{n\, |\log \vep|}.$$
\smallskip

\noindent Therefore, by \eqref{eq:sA},
\begin{eqnarray*}
\limsup_{n \, \to \, +\infty}\,\limsup_{\vep\, \to \, 0^+}\,\frac{\log \log S(\mathcal{P}(X), W_{1,n},\vep/4)}{n\, |\log \vep|} &\geqslant& \limsup_{n \, \to \, +\infty}\,\limsup_{\vep\, \to \, 0^+}\,\frac{\log \big(A(\mathcal{P}(X), d_n, \vep)\big)}{n \,|\log \vep|} \\
&\geqslant& \limsup_{n \, \to \, +\infty}\,\limsup_{\vep\, \to \, 0^+}\,\frac{\log S(X, d_n,\vep)}{n \,|\log \vep|}.    
\end{eqnarray*}
Consequently,
\begin{equation}\label{eq:2}
\limsup_{n \, \to \, +\infty}\,
\frac{\mathrm{\overline{mo}}\,\big(\mathcal{P}(X), W_{1,n}\big)}{n} \,\geqslant\, \limsup_{n \, \to \, +\infty}\,
\frac{\mathrm{\overline{dim}_B}\,\big(X,d_n\big)}{n}.    
\end{equation}
\smallskip

\noindent Bringing together \eqref{eq:1} and \eqref{eq:2}, we obtain
$$ \limsup_{n \, \to \, +\infty}\,
\frac{\mathrm{\overline{mo}}\,\big(\mathcal{P}(X), W_{1,n}\big)}{n} \,=\, \limsup_{n \, \to \, +\infty}\,
\frac{\mathrm{\overline{dim}_B}\,\big(X,d_n\big)}{n}.$$
\smallskip

We proceed by proving that 
\begin{equation}\label{eq:mbd}
\limsup_{n \, \to \, +\infty}\,
\frac{\mathrm{\overline{dim}_B}\,\big(X,d_n\big)}{n} \, \geqslant \, \mathrm{\overline{mdim}_M}\,\big(X,d,f\big).
\end{equation}
We may assume that  $\mathrm{\overline{dim}_B}\,\big(X,d_n\big) < +\infty$ for every $n \in \N$, otherwise \eqref{eq:mbd} is immediate due to \eqref{eq:dn}. As mentioned previously (see Subsections~\ref{sse:mdim} and \ref{sse:boxdim}), we can use $\mathcal{N}(X,n,\vep)$ instead of $S(X,n,\vep)$ to express these quantities. 

Denote by $\gamma$ the limit
$$\gamma \,=\, \limsup_{n \, \to \, +\infty}\,
\frac{\mathrm{\overline{dim}_B}\,\big(X,d_n\big)}{n} = 
 \limsup_{n \, \to \, +\infty}\, \frac{1}{n}\,
\left(\limsup_{\vep \, \to \, 0^+}\,\frac{\log  \mathcal{N}(X,n,\vep)}{|\log \vep|}\right).$$
Fix $\delta>0$ and take $N=N(\delta)\in\mathbb N$ so that 
$$
\frac{1}{N}\,\limsup_{\vep \, \to \, 0^+}\,\frac{\log  \mathcal{N}(X,N,\vep)}{|\log \vep|}\,<\, \gamma+\delta.$$
Now choose $\vep_0 = \vep_0(N) > 0$ such that, for every $0 < \vep < \vep_0$, one has
$$\frac{\log  \mathcal{N}(X,N,\vep)}{|\log \vep|} \, < \, N(\gamma+\delta).$$
Hence,
$$\log \mathcal{N}(X,N,\vep) \, < \, N(\gamma+\delta)\,|\log \vep|.$$
Fix $\vep$ such that $0 < \vep < \vep_0$. Then, since the sequence $\big(\log \mathcal{N}(X,n,\vep)\big)_{n\, \in\, \N}$ is subadditive, for all $k \in \N$ one has
$$\frac{\log \mathcal{N}(X,kN,\vep)}{kN} \, \leqslant \, \frac{k\,\log \mathcal{N}(X,N,\vep)}{kN} \,=\, \frac{\log \mathcal{N}(X,N,\vep)}{N} \, < \, (\gamma+\delta)\,|\log \vep|$$
and 
$$\limsup_{k \, \to \, +\infty}\,\frac{\log \mathcal{N}(X,kN,\vep)}{kN} \, = \, \lim_{m \, \to \, +\infty}\,\frac{\log \mathcal{N}(X,m,\vep)}{m}$$
since the latter limit exists. Therefore,
$$\lim_{m \, \to \, +\infty}\,\frac{\log \mathcal{N}(X,m,\vep)}{m} \, \leqslant \, (\gamma + \delta)\,|\log \vep|.$$
Thus,
$$\frac{\lim_{m \, \to \, +\infty}\,\frac{\log \mathcal{N}(X,\,m,\,\vep)}{m}}{|\log \vep|} \, \leqslant \, \gamma+\delta$$
and, as $\vep$ is arbitrary in $]0, \vep_0[$, 
$$ \limsup_{\vep \, \to \, 0^+}\,\frac{\lim_{m \, \to \, +\infty}\,\frac{\log \mathcal{N}(X,\,m,\,\vep)}{m}}{|\log \vep|} \, \leqslant \, \gamma+\delta.$$
Consequently,
$$\mathrm{\overline{mdim}_M}\,\big(X,d,f\big) \,= \, \limsup_{\vep \, \to \, 0^+}\,\frac{\lim_{m \, \to \, +\infty}\,\frac{\log \mathcal{N}(X,\,m,\,\vep)}{m}}{|\log \vep|} \, \leqslant\,  \gamma+\delta.$$
As $\delta>0$ is arbitrary, we conclude that $\mathrm{\overline{mdim}_M}\,\big(X,d,f\big) \,\leqslant \, \gamma.$ The proof of Theorem~\ref{thm3} is finished.

\subsection*{Acknowledgments} Maria Carvalho was partially supported by CMUP, member of LASI, which is financed by national funds through FCT – Fundação para a Ciência e a Tecnologia, I.P., under the project UID/00144/2025,  https://doi.org/10.54499/UID/00144/2025. Fagner B. Rodrigues  is grateful to CMUP and Faculdade de Ciências da Universidade do Porto for the excellent research conditions and hospitality.



\begin{thebibliography}{99}

\bibitem{AKM}
R. Adler, A. Konheim and M. McAndrew.
\newblock \emph{Topological entropy.}
\newblock Trans. Amer. Math. Soc. 114 (1965), 309--319



\bibitem{BCP}
A. Baraviera, M. Carvalho and G. Pessil. 
\newblock \emph{Metric mean dimension, H\"older regularity and Assouad dimension.} 
\newblock J. Fractal Geom. 2025. https://doi.org/10.4171/JFG/169

\bibitem{BSigmund}
W. Bauer and K. Sigmund.
\newblock \emph{Topological dynamics of transformations induced on the space of probability measures}.
\newblock Monatsh. Math. 79 (1975) 81--92.

\bibitem{Berger}
P. Berger.
\newblock \emph{Emergence and non-typicality of the finiteness of the attractors in many topologies}.
\newblock Proc. Steklov Inst. Math. 297:1 (2017) 1--27.

\bibitem{BB}
P. Berger and J. Bochi.
\newblock \emph{On emergence and complexity of ergodic decompositions}.
\newblock Adv. Math. 390 (2021), Paper No. 107904, 52 pp.

\bibitem{Bil}
P. Billingsley.
\newblock  Convergence of Probability Measures.
\newblock \emph{Wiley Series in Probability and Statistics}, John Wiley and Sons, New York, 2nd edition, 1999.


\bibitem{BGV}
F. Bolley, A. Guillin and C. Villani.
\newblock \emph{Quantitative concentration inequalities for empirical measures on non-compact spaces.}
\newblock Probab. Theory Related Fields 137 (2007) 541--593.


\bibitem{Bowen1}
R. Bowen.
\newblock\emph{Topological entropy and Axiom A.}
\newblock Proceedings of Symposia in Pure Mathematics, Global Analysis XIV, 1970.

\bibitem{BrinS}
M. Brin and G. Stuck.
\newblock \emph{Introduction to Dynamical Systems.}
\newblock Cambridge University Press, 2002.


\bibitem{BurguetShi}
D. Burguet and R. Shi.
\newblock \emph{Topological mean dimension of induced systems.}
\newblock Trans. Amer. Math. Soc. 378 (2025) 3085--3103.


\bibitem{CP2026}
M. Carvalho and G. Pessil.
\newblock \emph{On quantization and the classical variational principle for the metric mean dimension}.
\newblock Preprint arXiv:2603.17091v3, 2026.


\bibitem{CRV3}
M. Carvalho, F. B. Rodrigues and P. Varandas,
\newblock \emph{Generic homeomorphisms have full metric mean dimension}.
\newblock Ergodic Theory Dynam. Systems 142 (2020) 1--25.

\bibitem{CRV4}
M. Carvalho, F. B. Rodrigues and P. Varandas,
\newblock \emph{Topological and metric emergence of continuous maps.}
\newblock Math. Proc. Cambridge Philos. Soc. 177:3 (2024) 525–551.

\bibitem{CPV}
M. Carvalho, G. Pessil and P. Varandas. 
\newblock \emph{A convex analysis approach to the metric mean dimension: limits of scaled pressures and variational principles.}
\newblock Adv. Math. 436 (2024)  Paper No. 109407, 54 pp.

\bibitem{CYZ}
E. Chen, R. Yang and X. Zhou.
\newblock \textit{Measure-theoretic metric mean dimension.}
\newblock Studia Math. 280(1) (2025), 1–25.









\bibitem{Feng-Zhiying}
D. Feng, Z. Wen and J. Wu.
\newblock \emph{Some remarks on the box-counting dimensions.}
\newblock Progr. Natur. Sci. (English Ed.) 9:6 (1999) 409--415.

\bibitem{GlasnerWeiss}
E. Glasner and B. Weiss.
\newblock \emph{Quasi-factors of zero entropy systems.}
\newblock J. Amer. Math. Soc. 8 (1995) 665--686.

\bibitem{GraLus}
S. Graf and H. Luschgy.
\newblock Foundations of Quantization for Probability Distributions.
\newblock \emph{Lecture Notes in Math.} 1730, Springer-Verlag, Berlin, 2000.

\bibitem{Gro99}
M.~Gromov.
\newblock \emph{Topological invariants of dynamical systems and spaces of holomorphic maps I.}
\newblock Math. Phys. Anal. Geom. 2:4 (1999) 323--415.










\bibitem{GutmanSp}
Y. Gutman and A. \'Spiewak.
\newblock \emph{Metric mean dimension and analog compression.}
\newblock IEEE Trans. Inform. Theory 66:11 (2020), 6977--6998

\bibitem{Hazard}
P. Hazard.
\newblock \emph{Maps in dimension one with infinite entropy.}
\newblock  Ark. Mat. 58:1 (2020), 95--119.

\bibitem{KH}
A. Katok and B. Hasselblat.
\newblock \emph{Introduction to the Modern Theory of Dynamical Systems.} 
\newblock Encyclopedia of Mathematics and its Applications 54, Cambridge University Press, 1995.

\bibitem{KerrLi}
D. Kerr and H. Li.
\newblock \emph{Independence in topological and $C^*$-dynamics.}
\newblock Math. Ann. 338:4 (2007) 869--926.


\bibitem{Kolmogorov-Tikhomirov}
A. N. Kolmogorov and V. M. Tikhomirov.
\newblock\emph{$\epsilon$-Entropy and $\epsilon$-capacity of sets in functional spaces.}
\newblock Amer. Math. Soc. Transl. 2:17 (1961) 277--364.

\bibitem{KolyadaSnoha}
S. Kolyada and L. Snoha.
\newblock\emph{Topological entropy of nonautonomous dynamical systems.}
\newblock Random Comput. Dynam. 4 (1996), 205--233.




\bibitem{LW2000}
E. Lindenstrauss and B. Weiss.
\newblock \emph{Mean topological dimension.}
\newblock Israel J. Math. 115 (2000) 1--24.


\bibitem{LindTsu}
E. Lindenstrauss and M. Tsukamoto.
\newblock \emph{From rate distortion theory to metric mean dimension: variational principle.}
\newblock IEEE Transactions on Information Theory 64:5 (2018) 3590--3609.


\bibitem{Mis}
M.~Misiurewicz.
\newblock \emph{Horseshoes for continuous mappings of an interval.}
\newblock Dynamical Systems Lectures, C.I.M.E. Summer Schools 78, C. Marchioro (Ed.), Springer-Verlag Berlin Heidelberg, 2010, 127--135.














\bibitem{VV}
A. Velozo and R. Velozo.
\newblock \emph{Rate distortion theory, metric mean dimension and measure theoretic entropy.}
\newblock Preprint, 2017, arXiv:1707.05762.


\bibitem{Vi}
C. Villani.
\newblock Topics in Optimal Transportation.
\newblock \emph{Graduate Studies in Mathematics} 58, American Mathematical Society, Providence, RI, 2003.

\bibitem{Wa}
P. Walters.
\newblock An Introduction to Ergodic Theory.
\newblock  Springer-Verlag New York, 1982.

\bibitem{Yano}
K. Yano.
\newblock \emph{A remark on the topological entropy of homeomorphisms.}
\newblock Invent. Math. 59 (1980) 215--220.


\end{thebibliography}
\end{document}